\documentclass[12pt]{amsart}

\usepackage{amsmath}
\usepackage{amsthm}
\usepackage{xfrac}
\usepackage{txfonts}
\usepackage{indentfirst}
\usepackage{mathrsfs}
\usepackage{mathtools}
\usepackage{amssymb}

\newtheorem{theorem}{Theorem}[section]
\newtheorem{lemma}{Lemma}[section]

\newtheorem{definition}{Definition}[section]

\newtheorem{proposition}{Proposition}[section]
\newtheorem{corollary}{Corollary}[section]

\newtheorem{remark}{Remark}[section]

\newtheorem*{mainthm}{Main Theorem}
\newtheorem*{thmA}{Theorem A}
\newtheorem*{thmB}{Theorem B}
\newtheorem*{thmB'}{Theorem B'}
\newtheorem*{thmB''}{Theorem B''}

\numberwithin{equation}{section}

\newcommand{\R}{\mathbb{R}}
\newcommand{\Z}{\mathbb{Z}}
\newcommand{\eps}{\varepsilon}
\newcommand{\dist}{\textrm{dist}}
\newcommand{\E}{E}

\title{A Dichotomy for the Weierstrass-type functions}
\author{Haojie Ren and Weixiao Shen}
\date{\today}
\address{School of Mathematical Sciences, Fudan University, No 220 Handan Road, Shanghai, China 200433}
\address{Shanghai Center for Mathematical Sciences, Jiangwan Campus, Fudan University, No 2005 Songhu Road, Shanghai, China 200438}
\email{18210180009@fudan.edu.cn, wxshen@fudan.edu.cn}
\begin{document}
\setlength{\parindent}{1em}

\maketitle

\begin{abstract}
For a real analytic periodic function $\phi:\mathbb{R}\to \mathbb{R}$, an integer $b\ge 2$ and $\lambda\in (1/b,1)$, we prove the following dichotomy for the Weierstrass-type function
$W(x)=\sum\limits_{n\ge 0}{{\lambda}^n\phi(b^nx)}$: Either $W(x)$ is real analytic, or the Hausdorff dimension of its graph is equal to $2+\log_b\lambda$.
Furthermore, given $b$ and $\phi$, the former alternative only happens for finitely many $\lambda$ unless $\phi$ is constant.

\end{abstract}

\section{Introduction}
We study the fractal properties of the graphs of Weierstrass type functions
\begin{equation}\label{eqn:Wtype}
W(x)=W^{\phi}_{\lambda,b}(x)=\sum\limits_{n=0}^{\infty}{{\lambda}^n\phi(b^nx)},\,\, x \in \mathbb{R}
\end{equation}
where $b > 1$, $1/b< \lambda < 1$ and $\phi(x):\mathbb{R} \to \mathbb{R}$ is a non-constant $\mathbb{Z}$-periodic Lipschitz function. The most famous example, with $\phi(x)=\cos (2\pi x)$,
was introduced by Weierstrass, and it is a continuous nowhere differentiable function, see \cite{hardy1916weierstrass}. The graphs of Weierstrass-type and related functions are among the most studied objects in fractal geometry since the birth of this subject, see \cite{besicovitch1937sets}, \cite[Section 8.2]{Falconer} and~\cite[Chapter 5]{BP}, among many others.

The goal of this paper is to prove the following theorem.
\begin{mainthm} Let $b\ge 2$ be an integer, $\lambda\in (1/b,1)$ and let $\phi$ be a $\Z$-periodic real analytic function. Then exactly one of the following holds:
\begin{enumerate}
\item [(i)] $W$ is real analytic;
\item [(ii)] the graph of $W$ has Hausdorff dimension equal to
\begin{equation}\label{eqn:dfnD}
D=2+\log_b\lambda.
\end{equation}
\end{enumerate}
Moreover, given $b$ and  non-constant $\phi$, the first alternative only holds for finitely many $\lambda\in (1/b,1)$.
\end{mainthm}


Kaplan, Mallet-Paret and Yorke \cite{kaplan1984lyapunov} proved that in the case that $\phi$ is a trigonometric polynomial, either $W$ is a $C^1$ curve or the box dimension of the graph of $W$ is equal to $D$, without the assumption that $b$ is an integer.  Our theorem is a similar dichotomy with box dimension replaced by Hausdorff dimension which is much more difficult to compute. The price we pay here is the assumption that $b$ is an integer which enables us to approach the problem from dynamical point of view.

An immediate consequence is the following corollary which in particular recovers the main theorem in~\cite{shen2018hausdorff}.
\begin{corollary}\label{cor:cos} Let $b\ge 2$ be an integer, $\lambda\in (1/b,1)$ and let $\phi(x)=\cos (2\pi x+\theta)$, where $\theta\in \R$. Then the Hausdorff dimension of the graph of $W$ is equal to $D$.
\end{corollary}

\medskip
{\bf Historical remarks.}
A map $W$ as in (\ref{eqn:Wtype}) has the following remarkable property
\begin{equation}\label{eqn:selfaffineW}
W(x)=\phi(x)+\lambda W(bx),
\end{equation}
so the graph of $W$ exhibits approximate self-affinity with scales $b$ and $1/\lambda$, and it is natural to conjecture that the Hausdorff dimension of its graph is equal to $D$. However, one has to be careful since the function $W$ can be smooth for certain choices of $\lambda, b, \phi$. (This is easily seen: for any real analytic $\Z$-periodic function $W_0$ and $\phi(x)=W_0(x)-\lambda W_0(bx)$, one has $W_{\lambda,b}^\phi(x)=W_0(x)$.)
The pioneering work of Besicovitch and Ursell (\cite{besicovitch1937sets}) showed that the Hausdorff dimension of a function of the form $\sum_{n=0}^\infty b_n^{-\alpha} \phi( b_n x)$ is equal to $2-\alpha$ provided that $b_{n+1}/b_n\to\infty$ and $\log b_{n+1}/\log b_n\to 1$.
(See \cite{Baranski} for recent advances for maps of such modified form.) A map as in (\ref{eqn:Wtype}) is easily seen to be H\"older continuous of exponent $2-D$ which implies that the Hausdorff dimension of its graph is at most $D$. Many authors have studied the anti-H\"older property of these functions \cite{kaplan1984lyapunov,Reza, PU}, with the strongest form given in \cite{HL}, see Theorem \ref{thm:HL}. This anti-H\"older property implies that $W$ is not differentiable and also that the box and packing dimension of its graph are equal to $D$. Moreover, in \cite{PU}, it is proved that the Hausdorff dimension of the graph of such a $W$ is strictly greater one. In \cite{MW}, it was shown that the Hausdorff dimension of $W$ has a lower bound of the form $D-O(1/\log b)$.

The first example of maps in the form (\ref{eqn:Wtype}) for which the graph is shown to exactly have Hausdorff dimension $D$ was given by Ledrappier \cite{ledrappier1992dimension}. Using dimension theory for (non-uniformly) hyperbolic dynamical systems developed in \cite{ledrappier1985metric} and a Marstrand type projection argument, Ledrappier proved that the Hausdorff dimension of the graph of a Takagi function (taking $b=2$, $\phi(x)=\dist(x,\Z)$ in (\ref{eqn:Wtype})) is equal to $D$, provided that the Bernoulli convolution $\sum_n \pm (2\lambda)^{-n}$ has Hausdorff dimension one. The last property, studied first by Erd\"os~\cite{E}, was shown by Solomyak~\cite{Solo, PS} to hold for almost every $\lambda\in (1/2,1)$. More recently, it has been shown to hold for $\lambda$ outside a set of Hausdorff dimension zero in Hochman \cite{hochman2014self}. Mandelbrot \cite{Mandelbrot} conjectured that the Hausdorff dimension of the graph of $W$ is equal to $D$ for $\phi(x)=\cos(2\pi x)$ and all $\lambda\in (1/b,1)$.  Pushing Ledrappier's approach further, this conjecture has been proved for integral $b$, first for $\lambda$ close to $1$ in~\cite{baranski2014dimension} and then for all $\lambda\in (1/b,1)$ in~\cite{shen2018hausdorff}, in which a result of Tsujii~\cite{tsujii2001fat} also played an important role. See also \cite{Keller}.
The case $\phi(x)=\sin (2\pi x)$ has also been settled shortly after in \cite{Z}.

It had been known much earlier that the Bernoulli convolution has Hausdorff dimension less than one when $2\lambda$ is a Pisot number. So Ledrappier's approach has it limitation (as already pointed by himself).
Built upon the celebrated breakthrough~\cite{hochman2014self}, it has been shown recently in ~\cite{barany2019hausdorff} that the Hausdorff dimension of the graph of Takagi functions equal to $D$ for all $\lambda$, via analysis on entropy of convolutions of measures.

Let us mention that the box and Hausdorff dimensions of Weierstrass-type functions with random phases were obtained in respectively \cite{He} and \cite{Hunt}. See also \cite{Rom}.

See \cite{Baranskisurvey} and also \cite[Chapter 5]{BP} for more remarks on Weierstrass-type functions.

\medskip
{\bf Main findings.} We shall now be more technical and explain the main findings in this paper. Let $\mathbb{Z}_+$ denote the set of positive integers and let $\mathbb{N}$ denote the set of nonnegative integers.
Let $\varLambda=\{ 0,1,...,b-1 \}$,
$\varLambda^{\#}=\bigcup_{n=1}^{\infty} \varLambda ^n$ and $\Sigma=\varLambda^{\mathbb{Z}_+}$.
For $\textbf{j}=j_1j_2 j_3 \cdot \cdot \cdot  \in \Sigma $, define
\begin{equation}
Y(x,\textbf{j})=Y_{\lambda, b}^\phi(x,\textbf{j})=-\sum\limits_{n=1}^{\infty}{\gamma^{n}\phi^{\prime}\left(\frac x{b^n}+\frac{j_1}{b^n}+\frac{j_2}{b^{n-1}} + \cdot \cdot \cdot + \frac{j_n}b\right)},\,\, x \in \mathbb{R}
\label{Kernel}
\end{equation}
where
\begin{equation}\label{eqn:gamma}
\gamma=\frac{1}{b\lambda}\in \left(\frac{1}{b},1\right).
\end{equation}
This quantity appeared in \cite{ledrappier1992dimension} as the slopes of the strong stable manifolds of a dynamical system which has the graph of  $W|_{[0,1)}$ as an attractor. In both the approaches of ~\cite{ledrappier1992dimension} and~\cite{barany2019hausdorff}, certain separation properties of these functions $Y(x,\textbf{j})$ play an important role.

These functions $Y(x,\textbf{j})$ are indeed related to the Weierstrass-type function in a more direct way. Using the identity (\ref{eqn:selfaffineW})
one can show that if $W$ is Lipschitz, then $W'(x)=Y(x,\textbf{j})$ holds for Lebesgue a.e. $x\in \R$ and for any $\textbf{j}\in \Sigma$. In particular, we have $Y(x,\textbf{i})\equiv Y(x, \textbf{j})$ for all $\textbf{i}, \textbf{j}\in \Sigma$ in this case. See Lemma~\ref{lem:L2A}.

\begin{definition}
Given an integer $b\ge 2$ and $\lambda\in (1/b,1)$,	we say that a $\mathbb{Z}$-periodic $C^1$ function $\phi(x)$ satisfies
\begin{itemize}
\item the condition (H) if 
	    $$Y(x,\textbf{j})-Y(x,\textbf{i}) \nequiv 0, \quad \forall \, \textbf{j} \neq \textbf{i} \in \Sigma.$$
\item
the condition (H$^*$) if 
	    $$Y(x,\textbf{j})-Y(x,\textbf{i}) \equiv 0, \quad \forall \, \textbf{j},\,\, \textbf{i} \in \Sigma.$$
\end{itemize}
\end{definition}
Surprisingly, nothing happens between these two extreme cases.
\begin{thmA} Fix $b\ge 2$ integer and $\lambda\in (1/b,1)$. Assume that $\phi$ is $\Z$-periodic and $C^5$. Then exactly one of the following holds:
\begin{enumerate}
\item [(i)] $W_{\lambda,b}^{\phi}$ is $C^5$ and $\phi$ satisfies the condition (H$^*$);
\item [(ii)] $W_{\lambda, b}^{\phi}$ is not Lipschitz and $\phi$ satisfies the condition (H).
\end{enumerate}
\end{thmA}

To prove Theorem A, we introduce a concept called {\em $C^k$-regulating period} which is a real number $t$ for which $W(x+t)-W(x)$ is $C^k$.
A key estimate is that a positive $C^2$-regulating period $t$ is bounded from below in terms of the $C^2$-norm of $W(x+t)-W(x)$, provided that $W$ is not-Lipschitz. This is obtained from the anti-H\"older property established in \cite{HL, kaplan1984lyapunov}. See Lemma~\ref{lem:sizeqp}.

The proof of the main theorem is then completed by the following theorem and a theorem in~\cite{shen2018hausdorff}.
\begin{thmB}
	If a real analytic $\mathbb{Z}$-periodic function $\phi(x)$ satisfies the condition (H) for an integer $b\ge 2$ and $\lambda\in (1/b,1)$, then
		$$dim_H(\{(x,W_{\lambda,b}^{\phi}(x))\;|\;x\in[0,1)\})=D.$$
 \end{thmB}
Theorem B is obtained by modifying the argument of~\cite{barany2019hausdorff} where the dimension of planar self-affine measures are studied which in particular shows that the Hausdorff dimension of $W$ is equal to $D$ in the case $\phi(x)=\dist(x,\Z)$ and $b=2$. 
The strong separation property (H) and the real analytic assumption compensate the non-linearity we have to face.

Indeed, let $\mu$ denote the lift of the standard Lebesgue measure on $[0,1)$ to the graph of $W|_{[0,1)}$. By \cite{ledrappier1992dimension}, $\mu$ and its projections $\pi_{\textbf{j}}\mu$ along the strong unstable manifold of a dynamical system $F$ (which keeps the graph of $W|_{[0,1}$ invariant) are exact dimensional and that $\dim (\pi_{\textbf{j}}\mu)$ is equal to a constant $\alpha$ for typical $\textbf{j}\in\Sigma$, see \S\ref{subsec:Ledrappiertheory}. We need to show that $\alpha=1$. The measure $\pi_{\textbf{j}}\mu$ can be decomposed into measures of similar form in smaller scales, see (\ref{eqn:self`affine'}). Assuming the contrary, we shall apply Hochman's criterion on entropy increase (\cite{hochman2014self}) to obtain a contradiction.  An important step is to introduce a suitable sequence of partitions for the space $\mathcal{X}$ of the transformations involved, see (\ref{FunctionSet}). For the case $\phi=\dist(x, \Z)$, the set $\mathcal{X}$ is a subset of $\mathbb{A}_{2,1}$, the space of affine maps from $\R^2$ to $\R$, and a sequence of suitable partitions were constructed in \cite{barany2019hausdorff} using a rescaling-invariant metric in the space $\mathbb{A}_{2,1}$. Although we do not have such a metric in our nonlinear case, we deduce a strong separation property of maps in $\mathcal{X}$ from the condition (H) under the assumption that $\phi$ is real analytic, see \S\ref{sec:separation}.  With this strong separation property, we construct a sequence of partitions of $\mathcal{X}$ explicitly, see \S\ref{sec:partitionX}.

\begin{proof}[Proof of the Main Theorem] By Theorems A and B, we know that either (i) or (ii) holds.
To show the last statement, we apply Theorem from~\cite{shen2018hausdorff}, which asserts that for $\lambda$ close to $1/b$ (i.e. $\gamma$ close to $1$), the graph of $W_{\lambda, b}^\phi$ has Hausdorff dimension $D>1$. Assume by contradiction that there are infinitely many $\lambda_k\in (1/b, 1)$ such that $W_{\lambda_k, b}^\phi$ satisfies (i). Then $\lambda_k$ are bounded away from $1$ and
$$Y_{\lambda_k, b}^\phi(x, 000\cdots)\equiv Y_{\lambda_k, b}^\phi (x,100\cdots),$$ that is,
\begin{equation}\label{eqn:many=}
\sum_{n=1}^\infty \gamma^{n-1} \left(\phi'(x/b^n)-\phi'((x+1)/b^n)\right)=0
\end{equation}
for all $\lambda=\lambda_k$. For each $x\in \R$, the left hand side of (\ref{eqn:many=}) is a power series in $\gamma$ with radius of convergence at least one. It has infinitely many zeros compactly contained in the unit disk, so
$$\phi'(x/b^n)=\phi'((x+1)/b^n).$$
It follows that $\phi'$ is a constant, hence $\phi$ is a constant, a contradiction!
\end{proof}

\begin{proof}[Proof of Corollary~\ref{cor:cos}] By the Main Theorem, it suffices to show that $W$ is not real analytic. Arguing by contradiction, assume that $W$ is real analytic. Let $W(x)=\sum_{n\in \mathbb{Z}} a_n e^{2\pi i nx}$ be the Fourier series expansion of the $\mathbb{Z}$-periodic real analytic function $W$. Then $|a_n|$ is exponentially small in $|n|$. However, comparing the Fourier coefficients of both sides of the identity (\ref{eqn:selfaffineW}),
we obtain that $a_{b^k}=(\lambda^k+1)e^{i\theta}/2$ for all $k\ge 1$, absurd!
\end{proof}

\medskip
{\bf Problems.}
\begin{enumerate}
\item Let $b>1$ be non-integeral, $\lambda\in (1/b,1)$ and $\phi(x)=\cos(2\pi x)$. Does $W=W_{\lambda, b}^\phi$ have a $C^k$ regulating period, $1\le k\le \infty$? If the answer is yes and $T>0$ is a $C^k$ regulating period, then we can interpret the graph of $W|_{[0,T)}$ as an invariant repeller of the smooth dynamical system $(x,y)\mapsto (bx\mod T, (y-\cos(2\pi x))/\lambda +W(bx\mod T)-W(bx))$ and apply the corresponding dimension theory. If the answer is no, then it would be interesting to study the oscillation of the functions $W(x+T)-W(x)$ for $T>0$.
\item Is the $D$-dimensional Hausdorff measure of the graph of $W$ equal to zero, even assuming $b$ is an integer greater than one? It is well-known that the graph of $W|_J$, for any bounded interval $J$, has finite $D$-dimensional Hausdorff measure.

    In \cite{PU}, the case $\phi(x)$ the Rademacher function and $b=2$ were considered. That is
    $$\phi(x)=\left\{\begin{array}{ll}
    1 &\mbox{ if } \{x\}\in [0,1/2),\\
    -1 &\mbox{ if } \{x\}\in [1/2,1),
    \end{array}
    \right.
    $$
    where $\{x\}\in [0,1)$ denote the fractional part of $x$. In this case, it was proved that the $D$-dimensional Hausdorff measure of the graph of $W|_{[0,1)}$ is a positive real number if and only if the Bernoulli convolution $\sum_n\pm \lambda^n$ is absolutely continuous with respect to the Lebesgue measure and its density is in the class $L^\infty$. It is conceivable that for general $\phi$ and $b$, the problem is related to the joint essential boundedness of the densities of the occupation measures of $W(x)-\Gamma_{\textbf{u}}(x)$, $\textbf{u}\in \Sigma$. See \S\ref{sec:PreB} for the definition of $\Gamma_{\textbf{u}}$.
\end{enumerate}

\medskip
{\bf Organization.} We prove Theorem A in \S\ref{sec:thmA}. The rest of the paper is devoted to the proof of Theorem B. In \S\ref{sec:PreB}, we recall some results from the Ledrappier-Young theory and state Theorem B' which is a reduced form of Theorem B. The rest of the paper is then devoted to the proof of Theorem B' and an outline can be found at the end of \S\ref{subsec:transitionformlula}.

\medskip
{\bf Acknowledgment.} We would like to thank the participants of the dynamical systems seminar in the Shanghai Center for Mathematical Sciences, an in particular, Guohua Zhang for suggesting the name of regulating period. WS is supported by NSFC grant No. 11731003.

\section{The conditions (H) and (H$^\ast$)}\label{sec:thmA}
Throughout we fix an integer $b\ge 2$ and $\lambda\in (1/b, 1)$.
For a $\Z$-periodic and continuous function $\phi: \R\to \R$, define $W=W^{\phi}=W_{\lambda, b}^{\phi}$ as in (\ref{eqn:Wtype}).

\begin{theorem} \label{thm:dicho}
Assume that $\phi$ is $\Z$-periodic and of class $C^5$. Then exactly one of the following holds:
\begin{enumerate}
\item [(i)] $W$ is $C^5$ and $\phi$ satisfies the condition (H$^\ast$);
\item [(ii)] $W$ is not Lipschitz and $\phi$ satisfies the condition (H).
\end{enumerate}
\end{theorem}

\begin{remark}
At the cost of more technicality, the theorem can be proved under a weaker assumption that $\phi$ is $C^3$.
\end{remark}
The main idea of the proof is to analyze the regulating periods of $W$ defined as follows.

\begin{definition}
For each $k\in \mathbb{Z}_+$, we say that $t\in\R$ is a {\em $C^k$-regulating period} of $W=W^\phi$ if $W(x+t)-W(x)$ is a $C^{k}$ function. In this case, we put
\begin{equation}
\E_k(t)=\sup_{x\in \R} |(W(x+t)-W(x))^{(k)}|<\infty.
\end{equation}
%
\end{definition}
It is easy to see that for a given $k$ the set of all $C^k$-regulating periods of $W$ form an additive subgroup of $\R$. If $\phi$ is $C^k$, then every number of the form $mb^{-n}$, where $n\in \mathbb{Z}_+$ and $m\in \Z$, is a regulating period of $W$, which are called {\em trivial} regulating period.  If $W$ is $C^k$, then the subgroup is equal to $\R$.

It is fairly easy to show that if $W$ is Lipschitz and $\phi$ is $C^k$ then $W$ is $C^k$, and the condition (H$^\ast$) holds,
see Lemma~\ref{lem:L2A}.
Assuming that $W$ is not Lipschitz, we prove an lower bound of $|t|$ in terms of $E(t)$ (Lemma~\ref{lem:sizeqp}) and show that every $C^2$-regulating period is rational (Corollary~\ref{cor:qprational}).

Assuming by contradiction that $W$ is not Lipchitz and $\phi$ fails to satisfy the condition (H). We show that violation of the condition (H) yields a non-trivial regulating period of the form $1/p$, where $p$ is an integer greater than $1$ and co-prime with $b$. Given such an integer $p$, we define a renormalization (in \S\ref{subsec:ren}) of $\phi$ as follows:
$$\mathcal{R}_p(\phi)=\sum_{k\in \Z} c_{kp} e^{2\pi i k x},$$
where $c_m$ denotes the $m$-th Fourier coefficient of $\phi$. The properties that $W^{\phi}$ is not Lipschitz and $\phi$ does not satisfy the condition (H) are inherited by the renormalization $\mathcal{R}_p(\phi)$ (Proposition~\ref{prop:ren}). Hence we can repeat the procedure infinitely often. However, this would imply that $W$ is Lipschitz, a contradiction!

We start with the following easy observation.
\begin{lemma}\label{lem:L2A} If $W$ is Lipschitz and $\phi$ is $C^k$ for some $k\in \mathbb{Z}_+$, then $W$ is $C^k$ and $Y(x,\textbf{i})\equiv Y(x,\textbf{j})$ for all $\textbf{i}, \textbf{j}\in \Sigma$.
\end{lemma}
\begin{proof} Assume $W$ is Lipschitz. Then there exists a constant $C>0$ such that for Lebesgue a.e. $x\in \R$, $W'(x)$ exists and $|W'(x)|\le C$. From $W(x)=\phi(x)+\lambda W(bx)$, we obtain that
$$W'(x)=\phi'(x)+\gamma^{-1} W'(bx), a.e.$$
It follows that for a.e. $x\in \R$, if $(x_{-n})_{n=0}^\infty$ is a backward orbit of $x_0$ then for any $n\ge 0$, $W'(x_{-n})$ exists, $|W'(x_{-n})|\le C$, and
$$W'(x_{-n-1})=\phi'(x_{-n-1})+\gamma^{-1} W'(x_{-n}).$$
Therefore for any $\textbf{i}, \textbf{j}\in \Sigma$, $$Y(x, \textbf{i})=Y(x,\textbf{j})=-W'(x)$$
holds for a.e. $x\in \R$. 
Since $Y(x,\textbf{i})$ and $Y(x,\textbf{j})$ are $C^{k-1}$ functions, this implies that $Y(x,\textbf{i})\equiv Y(x,\textbf{j})$.  As $W(x)=W(0)-\int_0^x Y(t,\textbf{0}) dt$ is the integral of a $C^{k-1}$ function, $W$ is $C^k$.
\end{proof}

So we need to show that $\phi$ satisfies the condition (H) when $W$ is not Lipschitz. We shall use the following result due to Hu and Lau, see \cite[Theorem 4.1]{HL}. See also Kaplan, Mallet-Parret and York~\cite{kaplan1984lyapunov} for the case that $\phi$ is a trigonometric polynomial.
\begin{theorem}\label{thm:HL}
Assume that $\phi$ is Lipschitz but $W$ is not Lipschitz. Then there exists $c>0$ and $\kappa>0$ such that for any $\delta\in (0,1)$ and any $x\in \R$ there exists $y\in \R$ such that $c\delta< y-x<\delta$ and $|W(y)-W(x)|\ge \kappa |y-x|^\alpha$, where $\alpha=2-D$.
\end{theorem}

\subsection{Regulating periods}
\begin{lemma}[Key Estimate]\label{lem:sizeqp} Suppose that $\phi$ is Lipschitz but $W$ is not Lipschitz. Then there exist constants $t_0>0$ and $C_0>0$ such that
if $t\in \R\setminus\{0\}$ is a $C^2$-regulating period of $W$, then either $|t|>t_0$ or $E_2(t) |t|^D \ge C_0.$
\end{lemma}


\begin{proof} Let $c, \kappa$ be as in Theorem~\ref{thm:HL} and let $K>0$ be such that $|W(x)-W(y)|\le K|x-y|^\alpha$ for all $x, y\in \R$. We may assume that $t>0$. Let $$f(x)=W(x+t)-W(x)$$ and choose $x_0$ such that
$$|W(x_0+t)-W(x_0)|=\max_{x\in \R}|W(x+t)-W(x)|=:\Delta.$$
Note that $f'(x_0)=0$. Write $x_j=x_0+jt$ for each $j\in \mathbb{Z}$.


{\bf Claim.} There exist constants $t_1>0$ and $C>0$ such that either $t\ge  t_1$ or $\Delta\ge C t^\alpha$.

To prove this claim, fix a large positive integer $m$ such that
\begin{equation}\label{eqn:tlow1}
m^{\alpha}\ge 2K/(\kappa c^\alpha).
\end{equation}
Assume that $t< 1/m$. 
By Theorem~\ref{thm:HL}, there exists $y$  such that $c mt <y-x_0<mt$ and $|W(y)-W(x_0)|\ge \kappa |y-x_0|^\alpha$.
Let $m'$ be minimal such that $x_{m'}\ge y$. Then $cm<m'\le m$, and
\begin{multline}\label{eqn:Wtmt0low}
|W(x_{m'})-W(x_0)|\ge |W(y)-W(x_0)|-|W(y)-W(x_{m'})|\\
\ge \kappa c^\alpha m^{\alpha} t^{\alpha}- K t^{\alpha}\ge K t^{\alpha},
\end{multline}
where we have used (\ref{eqn:tlow1}) for the last inequality.
On the other hand, by maximality of $x_0$, we have
$$|W(x_{m'})-W(x_0)|=\left|\sum_{j=0}^{m'-1}(W(x_{j+1})-W(x_j))\right|\le m' \Delta\le m \Delta.$$
Together with (\ref{eqn:Wtmt0low}), this implies that
\begin{equation}\label{eqn:delta_0low}
\Delta\ge K t^{\alpha}/ m.
\end{equation}
Thus the claim holds with $t_1=1/m$ and $C=K/m$.

Now let us assume that $t< t_1$, so that $\Delta\ge Ct^\alpha$. Let $J=\left[\sqrt{\frac{C t^\alpha}{(E+1) t^2}}\right]$, where $E=E_2(t)$.
For any $0\le j\le J$,
since
$$f(x_j)-f(x_0)=\int_{x_0}^{x_j} \int_{x_0}^x f''(y) dy dx,$$
we obtain $$|f(x_j)-f(x_0)|\le \frac{E}{2} (x_j-x_0)^2= \frac{E}{2} j^2 t^{2}\le \frac{Ct^\alpha}{2}\le \frac{\Delta}{2},$$
and hence $|f(x_j)|\ge \Delta/2$ and $f(x_j)f(x_0)>0$.
Therefore,  for all $0\le k\le J$, 
$$|W(x_{k})-W(x_0)|=\sum_{j=0}^{k-1} |f(x_j)|\ge k \Delta/2.$$
Since
$$|W(x_J)-W(x_0)|\le K|x_J-x_0|^\alpha\le K J^\alpha t^\alpha,$$
we obtain
\begin{equation}\label{eqn:tlow3}
\Delta\le 2K t^\alpha/ J^{1-\alpha} .
\end{equation}
Together with $\Delta\ge C t^\alpha$, this implies that $J$ is bounded from above, hence $(E+1)t^{2-\alpha}$ is bounded away from zero. Thus either $t$ or $t^DE_2(t)$ is bounded away from zero.
%
\end{proof}

\begin{corollary} \label{cor:qprational}
If $\phi$ is Lipschitz but $W$ is not Lipschitz, then every $C^2$-regulating period of $W$ is rational.
\end{corollary}
\begin{proof}
Arguing by contradiction, assume that $W$ has a $C^2$-regulating period $t\in \R\setminus \mathbb{Q}$.
Then for each $n\ge 1$, $t_n:=\text{dist}(nt,\mathbb{Z})$ is a non-zero $C^2$-regulating period, and
$$E(t_n)=E(nt)\le nE(t),$$
so by Lemma~\ref{lem:sizeqp}, $|t_n|$ has a lower bound of the form $C n^{-1/D}$, where $D>1$ and $C>0$.
This contradicts with Dirichlet's theorem which asserts that for each irrational real number $t$ and any positive integer $Q$, there is an integer $q$ with $1\le q\le Q$ such that $\text{dist}(qt, \mathbb{Z})< 1/Q$.
\end{proof}

\begin{lemma}\label{lem:pexists} Assume that $\phi$ is $C^k$ for some integer $k\ge 2$ and does not satisfy the condition (H).  Assume also that $W$ is not Lipschitz. Then there is an integer $p>1$ such that $(p, b)=1$, and such that $1/p$ is a $C^k$-regulating period of $W$. 
\end{lemma}
\begin{proof}
Since $\phi$ does not satisfies the condition (H), there exist $\textbf{i}, \textbf{j}\in \Sigma$ with $\textbf{i}\not=\textbf{j}$
and such that $Y(x, \textbf{i})\equiv Y(x, \textbf{j})$. Without loss of generality, we may assume that $i_1\not=j_1$. Let $r_n=(i_1+i_2b+\cdots+i_n b^{n-1})/b^n$, $s_n=(j_1+j_2b+\cdots +j_n b^{n-1})/b^n$. Then $r_n\not=s_n$ for any $n\ge 1$.
For each $n$, $r_n-s_n$ and $t_n:=\text{dist}(r_n-s_n,\Z)$ are $C^k$-regulating periods of $W$. We first prove

{\bf Claim.} $\sup_{n=1}^\infty E_k(t_n)<\infty$.

Indeed, for each $n\ge 1$,
$$-Y(b^nx,\textbf{i})=\sum_{m=1}^\infty \gamma^{m} \phi'(b^{n-m}x +r_m)=\sum_{m=1}^n \gamma^{m} \phi'(b^{n-m} x+r_m)+\gamma^n\sum_{\ell=1}^\infty \gamma^{\ell} \phi'(b^{-\ell} x +r_{n+\ell}),$$
$$-Y(b^nx,\textbf{j})=\sum_{m=1}^\infty \gamma^{m} \phi'(b^{n-m}x +s_m)=\sum_{m=1}^n \gamma^{m} \phi'(b^{n-m} x+s_m)+\gamma^n\sum_{\ell=1}^\infty \gamma^{\ell} \phi'(b^{-\ell} x +s_{n+\ell}),$$
and hence
\begin{multline}\label{eqn:derexc}
\sum_{m=0}^n \gamma^{m} \phi'(b^{n-m} x+r_m)-\sum_{m=0}^n \gamma^{m} \phi'(b^{n-m} x+s_m)\\
=\gamma^n\left(\sum_{\ell=1}^\infty \gamma^{\ell} \phi'(b^{-\ell} x +s_{n+\ell})-\sum_{\ell=1}^\infty \gamma^{\ell} \phi'(b^{-\ell} x +s_{n+\ell})\right),
\end{multline}
where $r_0=s_0=0$.
Let
$$F_n(x)=\sum_{m=0}^{n} \lambda^m \phi(b^m (x+r_n))-\sum_{m=0}^{n} \lambda^m \phi(b^m(x+s_n))=W(x+r_n)-W(x+s_n).$$
Then,
\begin{align*}
F_n'(x)& =\sum_{m=0}^{n} \gamma^{-m} \phi'(b^m(x+r_n))-\sum_{m=0}^{n} \gamma^{-m} \phi'(b^m(x+s_n))\\
& =\gamma^{-n}\sum_{m=0}^{n} \gamma^{n-m} \left(\phi'(b^m x+r_{n-m})-\phi'(b^m x+s_{n-m}) \right)\\
& =\sum_{\ell=1}^\infty \gamma^{\ell} \left(\phi'(b^{-\ell} x+s_{n+\ell})-\phi'(b^{-\ell} x+r_{n+\ell})\right),
\end{align*}
where the second equality holds because for any $0\le m\le n$, $b^m r_n\equiv r_{n-m}, b^m s_n\equiv s_{n-m}\mod 1$, and the last equality follows from (\ref{eqn:derexc}).
As $E_k(t_n)=\sup_{x\in \R} |F_n^{(k)}(x)|$, it is bounded from above by a constant. The claim is proved.

Note that $t_n\not=0$, so by Lemma~\ref{lem:sizeqp}, $t_n$ is bounded away from zero. Now take $n_i\to\infty$ so that $r_{n_i}\to r$ and $s_{n_i}\to s$. As the proof of the claim shows, $F_n'$ lies in a compact family of $C^{k-1}$ functions, so $W(x+r)-W(x+s)$ is $C^k$. Therefore, $t=\text{dist}(r-s, \mathbb{Z})=\lim_{n_i\to\infty} t_{n_i}$ is a $C^k$-regulating period of $W$. By Corollary~\ref{cor:qprational}, $t\in \mathbb{Q}$. Since $t_n$ is bounded away from zero and for $n>m$,
$$b^m(r_{n}-s_{n})=(r_{n-m}-s_{n-m})\mod 1=\pm t_{n-m}\mod 1,$$
we obtain that $b^m(r-s)\not\in \Z$ for all integers $m\ge 0$. Therefore, $t$ does not have a finite $b$-adic expansion. So we can write $t$ in the form $q_1/p_1$ with $p_1\ge 1$, $(q_1, p_1)=1$ such that $p_1$ has a prime factor $p$ with $p\not\mid b$. As $W(x+1/p)-W(x)$ is a finite sum of translations of $W(x+r)-W(x+s)$, hence $C^k$, we obtain that $1/p$ is a $C^k$-regulating period of $W$.
\end{proof}
\subsection{Renormalization}\label{subsec:ren}
For each $C^5$ function $\phi:\R\to \R$ of period $1$ and any integer $p>1$, let
$$\widetilde{\mathcal{R}}_p \phi(x)=\sum_{k\in \Z} c_{kp} e^{2\pi i kp x},$$
and $$\mathcal{R}_p\phi(x) =\sum_{k\in \Z} c_{kp} e^{2\pi i k x}=\widetilde{\mathcal{R}}_p\phi(x/p),$$
where $c_k$ is the $k$-th Fourier coefficient of $\phi$. As $c_k=O(k^{-5})$, $\widetilde{\mathcal{R}}_p\phi(x)$ and $\mathcal{R}_p\phi(x)$ are $C^3$ functions.

Let $$\mathscr{P}(\phi)=\{p\in \Z_+: p>1, (p,b)=1, 1/p\mbox{ is a $C^3$-regulating period of } W^\phi\}.$$
For each $p\in\mathscr{P}(\phi)$, we call $\mathcal{R}_p\phi$ (resp. $\widetilde{\mathcal{R}}_p\phi$) a {\em renormalization} (resp. {\em pre-renormalzation}) of $\phi$.

The main properties of the renormalization is stated in the following proposition.
\begin{proposition} \label{prop:ren}
Assume that $\phi$ is $C^5$. Let $p\in \mathscr{P}(\phi)$. Then the following hold:
\begin{enumerate}
\item [(1)] For $\mathcal{S}_p(\phi)=\phi-\widetilde{R}_p\phi$, $W^{\mathcal{S}_p\phi}$ is $C^3$ and
$$\sup_{x\in \R} |(W^{\mathcal{S}_p\phi})'(x)|\le C,$$
where $C>0$ is a constant depending only on $\phi$.
\item [(2)] $W^{\phi}$ is Lipschitz if and only if $W^{\mathcal{R}_p\phi}$ is Lipschitz.
\item [(3)] $\phi$ satisfies the condition (H) if and only if so does $\mathcal{R}_p\phi$. 
\item [(4)] If $q\in \mathscr{P}(\mathcal{R}_p\phi)$ then $pq\in \mathscr{P}(\phi)$.
\end{enumerate}
\end{proposition}

We need a lemma to prove the proposition.
\begin{lemma}[Rescaling]\label{lem:rescale} Let $\tau$ be a $C^k$ function of period $1$ for some $k\in \mathbb{Z}_+$, and let $\tilde{\tau}(x)=\tau(px)$, where $p\ge 2$ is an integer with $(p,b)=1$. Then,
\begin{enumerate}
\item[(i)] $t$ is a $C^k$-regulating period of $W^{\tau}$ if and only if $\frac{t}{p}$ is a $C^k$-regulating period of $W^{\tilde{\tau}}$.
\item[(ii)] $\tau$ satisfies the condition (H) if and only if so does $\tilde{\tau}$.
\end{enumerate}
\end{lemma}
\begin{proof} It is straightforward to check that $W^{\tilde{\tau}}(x)=W^{\tau}(px)$ for all $x\in \R$, so
$$W^{\tilde{\tau}}(x+t/p)-W^{\tilde{\tau}}(x)=W^{\tau}(px+t)-W^{\tau}(px).$$ The statement (i) follows.
To prove (ii), we observe that for $\tilde{u}_j,\tilde{v}_j\in \{0,1,\ldots,b-1\}$, $j=1,2,\ldots$, there exists $u_j, v_j\in \{0,1,\ldots, b-1\}$, $j=1,2,\ldots$, such that
$$p(\tilde{u}_1+\tilde{u}_2 b+\cdots \tilde{u}_n b^{n-1})=u_1+u_2b +\cdots +u_{n} b^{n-1} \mod b^n,$$
$$p(\tilde{v}_1+\tilde{v}_2 b+\cdots \tilde{v}_n b^{n-1})=v_1+v_2b +\cdots +v_{n} b^{n-1} \mod b^n.$$
Vice versa, since $(p,b)=1$, given $u_1,u_2,\cdots, v_1, v_2,\cdots$, we can find $\tilde{u}_1,\tilde{u}_2,\cdots, \tilde{v}_1,\tilde{v}_2,\cdots$ so that the above properties hold. Moreover, $u_1u_2\cdots=v_1v_2\cdots$ if and only if $\tilde{u}_1\tilde{u}_2\cdots=\tilde{v}_1\tilde{v}_2\cdots$.
Since
$$Y^{\tilde{\tau}}(x, \tilde{u}_1\tilde{u}_2\cdots)-Y^{\tilde{\tau}}(x,\tilde{v}_1\tilde{v}_2\cdots)=p(Y^\tau(px, u_1u_2\cdots)-Y^\tau(px, v_1v_2\cdots)),$$
the statement follows.
\end{proof}

\begin{proof}[Proof of Proposition~\ref{prop:ren}] (1)  Note $W^\phi(x)=W^{\widetilde{\mathcal{R}}_p\phi}(x)+W^{\mathcal{S}_p\phi}(x)$.
For each $m\in\Z$, let $a_m=\int_0^1 W^\phi(x) e^{-2\pi i mx} dx$ be the $m$-th Fourier coefficient of $W^\phi(x)$.
Note that
$$
\int_0^1 W^{\widetilde{R}_p\phi}(x) e^{-2\pi i mx} dx
=\left\{
\begin{array}{ll}
a_m & \mbox{ if } p\mid m\\
0  &\mbox{ if } p\not\mid m,
\end{array}
\right.
$$
and
$$
\int_0^1 W^{\mathcal{S}_p\phi}(x) e^{-2\pi i mx} dx=\left\{
\begin{array}{ll}
0 &\mbox{ if } p\mid m,\\
a_m &\mbox{ if } p\not\mid m.
\end{array}
\right.$$
Thus for all $m\in \Z$,
$$(1-e^{2\pi i m/p})\int_0^1 W^\phi (x) e^{-2\pi imx} dx=(1-e^{2\pi i m/p})\int_0^1 W^{\mathcal{S}_p\phi}(x) e^{-2\pi imx} dx,$$
i.e.,
$$\int_0^1 (W^\phi(x+1/p)-W^\phi(x))e^{-2\pi im x} dx= \int_0^1 (W^{\mathcal{S}_p\phi}(x+1/p)-W^{\mathcal{S}_p\phi}(x)) e^{-2\pi i mx} dx.$$
Therefore $W^{\mathcal{S}_p\phi}(x+1/p)-W^{\mathcal{S}_p\phi}(x)=W^\phi (x+1/p)-W^\phi (x)$ is $C^3$.
Furthermore, $|a_m|$, $p\not\mid m$, is of order $m^{-3}$. For $p\mid m$, the $m$-th Fourier coefficient of $W^{\mathcal{S}_p\phi}$ is zero, and for $p\not\mid m$, it is $a_m$. Thus $W^{\mathcal{S}_p\phi}(x)$ is $C^1$.
As in Lemma~\ref{lem:L2A}, we have 
$(W^{\mathcal{S}_p\phi})'(x)=-\sum_{n=1}^\infty \gamma^n (\mathcal{S}_p\phi)'(x/b^n)$. So $W^{\mathcal{S}_p\phi}$ is $C^3$. Since
$$|(\mathcal{S}_p\phi)'(y)|=|\sum_{p\not\mid m} c_m m e^{2\pi imy}|\le \sum_{m\in \Z} |mc_m|=:C_0<\infty,$$
$\sup_x |(W^{\mathcal{S}_p\phi})'(x)|\le C_0\gamma/(1-\gamma)=:C$.

(2) Since $W^{\widetilde{\mathcal{R}}_p\phi}(x)=W^{\mathcal{R}_p\phi}(px)$, $W^{\mathcal{R}_p\phi}(x)$ is Lipschitz if and only if so is $W^{\widetilde{\mathcal{R}}_p\phi}$. By (1), $W^\phi(x)-W^{\widetilde{\mathcal{R}}_p\phi}(x)$ is Lipschitz. So the statement holds.

(3) Since $W^{\mathcal{S}_p\phi}(x)=W^\phi(x)-W^{\widetilde{\mathcal{R}}_p\phi}(x)$ is Lipschitz, by Lemma~\ref{lem:L2A}, $Y^{\mathcal{S}_p\phi}(x,\textbf{i})\equiv Y^{\mathcal{S}_p\phi}(x,\textbf{j})$ for any $\textbf{i}, \textbf{j}\in \Sigma$. So $\phi$ satisfies the condition (H) if and only if so does $\widetilde{\mathcal{R}}_p\phi$. By Lemma~\ref{lem:rescale} (ii), $\widetilde{\mathcal{R}}_p\phi$ satisfies the condition (H) if and only if so does ${\mathcal{R}}_p\phi$.

(4) Since $\mathcal{R}_p\phi$ is $C^3$, by Lemma~\ref{lem:rescale} (i), $pq\in \mathscr{P}(\widetilde{\mathcal{R}}_p\phi)$. Since $W^{\mathcal{S}_p\phi}$ is $C^3$, this implies that $pq\in \mathscr{P}(\phi)$.
\end{proof}

We shall now complete the proof of Theorem~\ref{thm:dicho}.
\begin{proof}[Completion of proof of Theorem~\ref{thm:dicho}]
If $W$ is Lipschitz, then (i) holds by Lemma~\ref{lem:L2A}. Assume now that $W$ is not Lipschitz and let us prove that (ii) holds. Arguing by contradiction, assume that $\phi$ does not satisfies the condition (H).
By Lemma~\ref{lem:pexists}, $\mathscr{P}(\phi)$ is not empty.

Given $p\in\mathscr{P}(\phi)$, by Proposition~\ref{prop:ren}, $W^{\mathcal{R}_p\phi}$ is not Lipschitz and $\mathcal{R}_p\phi$ does not satisfies the condition (H). So by Lemma~\ref{lem:pexists}, there is $q\in \mathscr{P}(\mathcal{R}_p\phi)$. By Proposition~\ref{prop:ren} (4), $pq\in \mathscr{P}(\phi)$. By definition, $p, q\ge 2$, so $pq>p$. Therefore, $\mathscr{P}(\phi)$ is an infinite set.

Let $p_1<p_2<\cdots$ be the elements of $\mathscr{P}$. Then clearly $W^{\widetilde{\mathcal{R}}_{p_k}\phi}(x)\to 0$ holds for all $x\in \R$. By Proposition~\ref{prop:ren} (1), $\sup|(W^{\widetilde{\mathcal{R}}_{p_k}\phi}-W)'(x)|\le C$. It follows that $W$ is Lipschitz, a contradiction!
\end{proof}

\section{Preliminaries for  the proof of Theorem B}\label{sec:PreB}
In the remainder of the paper, we shall prove Theorem B. So fix an integer $b\ge 2$ and $\lambda\in (1/b,1)$ and assume that $\phi$ is a $\mathbb{Z}$-periodic analytic function which satisfies the condition (H).
We start with recalling some basic facts from the Ledrappier-Young theory.

A probability measure $\omega$ in a metric space $X$ is called {\em exact-dimensional} if there exists a constant $\alpha\ge 0$ such that for $\omega$-a.e. $x$,
$$\lim_{r\to 0} \frac{\log \omega(B(x,r))}{\log r}=\alpha.$$
In this case, we write $\dim \omega=\alpha$.
By the mass distribution principle, this implies that for any Borel subset $E$ of $X$ with $\omega(E)>0$, we have $\dim_H(E)\ge \alpha.$

\subsection{Ledrappier's Theorem}\label{subsec:Ledrappiertheory}
Let $\mu$ denote the pushforward of the Lebesgue measure in $[0,1)$ to the graph of $W$ by $x\mapsto (x, W(x))$. To complete the proof of Theorem B, it suffices to show that $\dim (\mu)\ge D$, since it is well-known that $W(x)$ is a $C^{2-D}$ function and hence the Hausdorff dimension of its graph is at most $D$. 

The graph of $W|_{[0,1)}$ is invariant under the dynamical system
$$F: [0,1)\times \R\to [0,1)\times \R, \, (x,y)\mapsto \left(bx \mod 1, \frac{y-\phi(x)}{\lambda}\right)$$
and $\mu$ is an invariant probability measure. The Ledrappier-Young's dimension theory of dynamical systems applies in this setting, which
relates the dimension of $\mu$ with its projection along some dynamical defined flows. We shall now recall the results obtained in Ledrappier~\cite{ledrappier1992dimension}.

As before let $\varLambda=\{0,1,\ldots, b-1\}$ and let $\Sigma=\varLambda^{\mathbb{Z}_+}$.  Let $\sigma: \Sigma\to\Sigma$ denote the shift map $(i_1i_2\cdots)\mapsto (i_2i_3\cdots)$. Let $\nu$ denote the even distributed probability measure on $\varLambda$ and let $\nu^{\mathbb{Z}_+}$ denote the product (Bernoulli) measure on $\Sigma$.
For each $i\in \varLambda$, define
\begin{equation}
\label{IFS}
 g_i(x,y)=\bigg(\frac{x+i}b, \lambda y +\phi\left(\frac{x+i}b\right)\bigg).
\end{equation}
Define the `inverse' of $F$ as
$$G:[0,1)\times \R \times \Sigma\to [0,1)\times \R\times \Sigma,\,\, (x,y,\textbf{i})\mapsto (g_{i_1}(x,y),\sigma(\textbf{i})).$$
Then
\begin{equation}\label{eqn:muinv}
\mu=\frac{1}{b}\sum_{i=0}^{b-1} g_i \mu.
\end{equation}
Direct computation shows that
$$Dg_{i_1}(x,y)\left(\begin{array}{ll}
 1\\
 Y(x, \textbf{i})
\end{array} \right)=\frac{1}{b}\left(
\begin{array}{ll}
1\\
Y((x+i_1)/b, \sigma(\textbf{i}))
\end{array}
\right).$$
So $Dg_{i_n}g_{i_{n-1}}\cdots g_{i_1}$ contracts the vector $(1, Y(x,\textbf{i}))$ at the exponential rate $-\log b$.
Let $$\Gamma_{\textbf{i}}(x)=\int_0^x Y(t,\textbf{i})dt.$$
So for each $y$, $x\mapsto y+\Gamma_{\textbf{i}}(x)$ is the integral curve of the vector filed $(1,Y(x,\textbf{i}))$ which passes through $(0,y)$. For each $\textbf{i}\in \Sigma$, this defines a foliation in $[0,1)\times \R$ whose leaves are ``parallel'' to each other.
For $\textbf{i} \in \Sigma$,  define
\begin{equation}
\label{ProjectionFunction}
\pi_{\textbf{i}}(x,y) = y -\Gamma_{\textbf{i}}(x),
\,\,\, (x,y) \in [0,1)\times \mathbb{R}.
\end{equation}
So $\pi_{\textbf{i}}$ is the projection of $(x,y)$ into the line $x=0$ along the foliation $\{y+\Gamma_{\textbf{i}}(x)\}_{y\in \R}$.
We call $\pi_{\textbf{i}}$ the {\em flow projection function} with respect to $\textbf{i}$.

The following result is a part of \cite[Proposition 2]{ledrappier1992dimension} which
serves as our starting point to calculate the Hausdorff dimension of the graph of $W$.

\begin{theorem}\label{thm:ledrappier}
	If $\phi:\mathbb{R} \to \mathbb{R}$ is a $\mathbb{Z}$-periodic continuous piecewise $C^2$ function, then
	\begin{enumerate}
	\item[(1)] $\mu$ is exact dimensional;
	\item[(2)] there is a constant $\alpha\in [0,1]$ such that for $\nu^{\mathbb{Z}_+}$-a.e. $\textbf{i}\in \Sigma$, $\pi_{\textbf{j}}\mu$ is exact dimensional and $\dim(\pi_{\textbf{j}}\mu)=\alpha$.
	\item[(3)]
    \begin{equation}
	\label{LedrapperYoungF}
	\dim(\mu)=1+(D-1)\alpha.
	\end{equation}
    \end{enumerate}
\end{theorem}

Therefore, Theorem B is reduced to the following
\begin{thmB'} Fix an integer $b\ge 2$ and $\lambda\in (1/b,1)$. Assume that $\phi$ is a real analytic $\Z$-periodic function which satisfies the condition (H). Then $\alpha=1$, where $\alpha$ is the constant in Theorem~\ref{thm:ledrappier}.
\end{thmB'}

\subsection{A transition formula}\label{subsec:transitionformlula}
We shall follow the strategy in \cite{barany2019hausdorff}, built on~\cite{hochman2014self}, to prove Theorem B'. For $\textbf{i}=i_1i_2\cdots i_n\in \varLambda^n$, write $g_{\textbf{i}}=g_{i_1}\circ g_{i_2}\circ \cdots g_{i_n}$. By iterating the formula (\ref{eqn:muinv}), we obtain
$$\mu=\frac{1}{b^n}\sum_{\textbf{i}\in \varLambda^n}\ g_{\textbf{i}}\mu$$
and hence for each $\textbf{j}\in \Sigma$, $\pi_{\textbf{j}}\mu$ decomposes into measures on small scales as
\begin{equation}\label{eqn:self`affine'}
\pi_{\textbf{j}}\mu=\frac{1}{b^n}\sum_{\textbf{i}\in \varLambda^n} \pi_{\textbf{j}}\circ g_{\textbf{i}} \mu.
\end{equation}
This resembles the case of self-similar/self-affine measures, as the maps $\pi_{\textbf{j}} g_{\textbf{i}}$ satisfies the following transition rule, which implies that each of the measure in the right hand side of (\ref{eqn:self`affine'}) is a translated rescaling of a measure of the form $\pi_{\textbf{i}}\mu$.

Recall that $\varLambda^\#=\bigcup_{n=1}^\infty \varLambda^n$. For each $\textbf{i}=i_1i_2\cdots i_n\in \varLambda^\#$, set $|\textbf{i}|=n$ and 
\begin{equation}\label{eqn:bfi'}
\textbf{i}^*=i_ni_{n-1}\cdots i_1.
\end{equation}
\begin{lemma}\label{lem:TransformA}
For any $\textbf{j}\in\Sigma$ and $\textbf{i}\in\varLambda^\#$,
\begin{equation}
\label{TransformA}
\pi_{\textbf{j}}\,g_{\textbf{i}}\,(x,y)=\lambda^{|\textbf{i}|}\,\pi_{{\textbf{i}}^*\textbf{j}}(x,y)+\pi_{\textbf{j}}g_{\textbf{i}}\,(0)
\end{equation}
\end{lemma}
\begin{proof}
	By induction it suffices to consider the case $\textbf{i}=i\in\varLambda$.
	According to definition, we have
\begin{multline*}	
	\pi_{\textbf{j}}g_i(x,y)=\pi_{\textbf{j}}\left(\frac{x+i}{b},\lambda y+\phi\left(\frac{x+i}{b}\right)\right)\\
=\lambda y+\phi\left(\frac{x+i}{b}\right)+\int_0^{\frac{x+i}{b}}\sum\limits_{n=1}^{\infty}{\gamma^{n}\phi^{\prime}\left(\frac s{b^n}+\frac{j_1}{b^n} + \cdot \cdot \cdot + \frac{j_n}b\right)} ds\\
		=\lambda y+\lambda\int_0^x\gamma\phi^{\prime}(\frac{u+i}{b})du+\lambda\int_0^x\sum\limits_{n=1}^{\infty}{\gamma^{n+1}\phi^{\prime}\left(\frac u{b^{n+1}}+\frac i{b^{n+1}} + \cdot \cdot \cdot + \frac{j_n}b\right)}du+\pi_{\textbf{j}}g_{\textbf{i}}\,(0).
\end{multline*}	
\end{proof}

To apply the argument in~\cite{barany2019hausdorff}, we need to show the following:
\begin{enumerate}
\item [(i)] Most of the measures in the right hand side of (\ref{eqn:self`affine'}) has certain {\em entropy porous} property. This will be done in \S\ref{sec:entpor} and is similar to the corresponding part of \cite{barany2019hausdorff}.
\item [(ii)] Maps in the space
\begin{equation}
\label{FunctionSet}
\mathcal{X} =\{\,\pi_{\textbf{j}} \circ g_{\textbf{i}}\;\big| \;\textbf{j}\in \Sigma,\,\; \textbf{i} \in \Lambda^{\#}\},
\end{equation}
satisfy a suitable separation condition. This will be done in \S\ref{sec:separation} and our argument uses essentially the real analytic assumption on $\phi$. This separation property enables us to define a sequence of suitable partitions of $\mathcal{X}$ in \S\ref{sec:partitionX}.
\end{enumerate}
After these preparations, the proof of Theorem B' will be given in \S\ref{sec:pfThmB'}.

\subsection{Entropy of measures}
We shall recall definition and basic properties of entropy of measures which is a basic tool for the proof of Theorem B'.

Consider a probability space $(\Omega, \mathcal{B}, \omega)$. A {\em (countable) partition} $\mathcal{Q}$ is a countable collection of pairwise disjoint measurable subsets of $\Omega$ whose union is equal to $\Omega$. We use $\mathcal{Q}(x)$ to denote the member of $\mathcal{Q}$ which contains $x$. If $\omega(\mathcal{Q}(x))>0$, then we call the conditional measure
$$\omega_{\mathcal{Q}(x)}(A)=\omega_{x,\mathcal{Q}}(A)=\frac{\omega(A\cap \mathcal{Q}(x))}{\omega(\mathcal{Q}(x))}$$
a {\em $\mathcal{Q}$-component} of $\omega$.
We define the {\em entropy}
$$H(\omega, \mathcal{Q})=\sum_{Q\in\mathcal{Q}} -\omega(Q) \log_b \omega(Q),$$
where the common convention $0\log 0=0$ is adopted.
Given another countable partition $\mathcal{P}$, we define the {\em condition entropy} as
$$H(\omega, \mathcal{Q}|\mathcal{P})=\sum_{P\in \mathcal{P}, \omega(P)>0} \omega(P) H(\omega_P, \mathcal{Q}).$$
When $\mathcal{Q}$ is a {\em refinement} of $\mathcal{P}$, i.e.,
$\mathcal{Q}(x)\subset\mathcal{P}(x)$ for each $x\in \Omega$, we have
$$H(\omega, \mathcal{Q}|\mathcal{P})=H(\omega, \mathcal{Q})-H(\omega, \mathcal{P}).$$

We shall consider the case where there is a sequence of partitions $\mathcal{Q}_i$, $i=1,2,\cdots$, such that $\mathcal{Q}_{i+1}$ is a refinement of $\mathcal{Q}_i$. In this situation, we shall write $\omega_{x,i}=\omega_{x,\mathcal{Q}_i}$, and call it a {\em $i$-th component measure of } $\omega$. For a finite set $I$ of positive integers, suppose that for each $i\in I$, there is a random variable $f_i$ defined over $(\Omega, \mathcal{B}(\mathcal{Q}_i), \omega)$, where $\mathcal{B}(\mathcal{Q}_i)$ is the sub-$\sigma$-algebra of $\mathcal{B}$ which is generated by $\mathcal{Q}_i$. Then we shall use the following notation
$$\mathbb{P}_{i\in I} (B_i)=\mathbb{P}_{i\in I}^\omega(B_i):=\frac{1}{\# I} \sum_{i\in I} \omega(B_i),$$
where $B_i$ is an event for $f_i$. If $f_i$'s are $\mathbb{R}$-valued random variable, we shall also use the notation
$$\mathbb{E}_{i\in I} (f_i)=\mathbb{E}^\omega_{i\in I}(f_i):=\frac{1}{\# I} \sum_{i\in I} \mathbb{E}(f_i).$$
For example, we have
$$H(\omega, \mathcal{Q}_{m+n}|\mathcal{Q}_n)=\mathbb{E}(H(\omega_{x, n}, \mathcal{Q}_{m+n}))=\mathbb{E}_{i=n} (H(\omega_{x,i},\mathcal{Q}_{i+m})).$$
These notations were used extensively in ~\cite{hochman2014self} and \cite{barany2019hausdorff}.

In particular, we shall often consider the case $\Omega=\mathbb{R}$ and $\mathcal{B}$ the Borel $\sigma$-algebra. Let $\mathcal{L}_n$ denote the partition of $\mathbb{R}$ into $b$-adic intervals of level $n$, i.e., the intervals $[j/b^n, (j+1)/b^n)$, $j\in \mathbb{Z}$. Let $\mathscr{P}(\R)$ denote the collection of all Borel probability measures in $\R$. For an exact dimensional probability measure $\omega\in \mathscr{P}(\mathbb{R})$, its dimension is closely related to the entropy, as shown in the following fact which is
\cite[Theorem 4.4]{Young}. See also~\cite[Theorem 1.3]{FLR}.
\begin{proposition}\label{prop:Young}
	If $\omega \in \mathscr{P}(\mathbb{R})$ is exact dimensional, then
		$$\dim(\omega)= \lim\limits_{n\to\infty}\frac{1}{n} H(\omega,\mathcal{L}_n).$$
\end{proposition}

These notations $\mathbb{P}_{i\in I}(B_i)$, $\mathbb{E}_{i\in I} (f_i)$ will also apply to the case where $\Omega=\mathcal{X}$, $\mathcal{B}$ is the collection of all subsets of $\mathcal{X}$, and $\omega$ is a discrete measure.

In the following, we collect a few well-known facts about entropy and conditional entropy.

\begin{lemma}[Concavity]\label{lem:concave}
Consider a measurable space $(\Omega, \mathcal{B})$ which is endowed with partitions $\mathcal{Q}$ and $\mathcal{P}$ such that $\mathcal{P}$ is a refinement of $\mathcal{Q}$. Let $\omega, \omega'$ be probability measures in $(\Omega, \mathcal{B})$. The for any $t\in (0,1)$,
$$tH(\omega,\mathcal{Q})+(1-t)H(\omega',\mathcal{Q})\le H(t\omega+(1-t)\omega',\mathcal{Q}),$$
$$tH(\omega,\mathcal{P}|\mathcal{Q})+(1-t)H(\omega',\mathcal{P}|\mathcal{Q})\le H(t\omega+(1-t)\omega',\mathcal{P}|\mathcal{Q}).$$
\end{lemma}

\begin{lemma} \label{lem:affinetransform}
Let $\omega\in \mathcal{P}(\mathbb{R})$. There is a constant $C>0$ such that for any affine map  $f(x)=ax+c$, $a, c\in \R$, $a\not=0$ and for any $n\in\mathbb{N}$ we have
$$\left|H(f\omega,\,\mathcal{L}_{n+[\log_b |a|]})-H(\omega,\,\mathcal{L}_{n})\right|\le C.$$
\end{lemma}
%

\begin{lemma}\label{lem:pfclose}
Given a probability space $(\Omega, \mathcal{B}, \omega)$, if $f,g:\Omega\to\mathbb{R}$ are measurable and $\sup_x|f(x)-g(x)|\le b^{-n}$ then
$$\left|H(f\omega, \mathcal{L}_{n})-H(g\omega, \mathcal{L}_{n})\right|\le C,$$
where $C$ is an absolute constant.
\end{lemma}

%
%

\section{Entropy porosity}\label{sec:entpor}
This section is devoted to analysis of entropy porosity of the projected measures $\pi_{\textbf{j}}\mu$. This property will be used in applying Hochman's criterion to obtain entropy growth under convolution.

 \begin{definition}[Entropy porous]
 	Let $\omega \in \mathscr{P}(\mathbb{R})$. We say that $\omega$ is $(h,\delta,m)$-entropy porous from scale $n_1$ to $n_2$ if
 	$$\mathbb{P}^\omega_{n_1\le i \le n_2} \left(\frac{1}{m} H (\omega_{x,i},\mathcal{L}_{i+m})<h+\delta\right) >1-\delta.$$
 \end{definition}

The main result of this section is the following Theorem~\ref{thm:entporous}. Before the statement of the theorem, we need to introduce a notation.


{\bf Notation.} For each integer $n\ge 0$, let $\hat{n}$ be the unique integer such that
\begin{equation}\label{eqn:n'}
\lambda^{\hat{n}}\le b^{-n}< \lambda^{\hat{n}-1}.
\end{equation}

In particular, $\hat{0}=0$. With this notation, there is a constant $C_0>0$ such that for any $\textbf{j}\in \Sigma$, $\textbf{i}\in \varLambda^{\hat{n}}$ and any $m\in \mathbb{N}$,
\begin{equation}\label{eqn:jiij}
\left|H(\pi_{\textbf{j}}g_{\textbf{i}}\mu, \mathcal{L}_{n+m})-H(\pi_{{\textbf{i}}^*\textbf{j}}\mu, \mathcal{L}_m)\right|\le C_0.
\end{equation}
Indeed, by Lemma~\ref{lem:TransformA}, $\pi_{\textbf{j}}g_{\textbf{i}}\mu$ is equal to the pushforward of $\pi_{{\textbf{i}}^*\textbf{j}}\mu$ by a map $\lambda^{|\textbf{i}|}x+c$, for some $c\in \R$. So the statement follows from Lemma~\ref{lem:affinetransform}.

\begin{theorem}\label{thm:entporous}
Fix an integer $b\ge 2$ and $\lambda\in (1/b,1)$. Assume that $\phi:\mathbb{R} \to \mathbb{R}$ is a $\mathbb{Z}$-periodic piecewise $C^2$ function such that $W=W_{\lambda,b}^\phi$ is not a Lipschitz function. Then
for any $\eps>0$, $m\ge M(\eps),$ $k\ge K(\eps,m)$ and $n\ge N(\eps,m,k)$, the following holds:	For any $\textbf{j}\in\varSigma$ and $\textbf{u}\in \varLambda^{\hat{t}}$, $t\in \mathbb{N}$,
	$$\nu \left(\left\{\textbf{i}=(i_1i_2\cdots)\in \Sigma:
	\begin{matrix}
	\pi_{\textbf{j}}g_{\textbf{u}} g_{i_1i_2\cdots i_{\hat{n}}}\mu \mbox{ is } (\alpha,\epsilon,m)-\mbox{entropy} \\
	 \mbox{ porous from scale } t+n+1\mbox{ to }t+n+k
	\end{matrix}
	\right\}\right) > 1-\varepsilon.
	$$
\end{theorem}

We shall follow the argument in~\cite[Section 3]{barany2019hausdorff} to prove this theorem. In particular, we shall use (\ref{eqn:self`affine'}) to decompose a measure $\pi_{\textbf{j}}\mu$ as a convex combination of measures of the form $\pi_{\textbf{j}} g_{\textbf{i}}\mu$.

%
%

\subsection{Uniform continuity across scales}
Following~\cite{barany2019hausdorff}, we say that a measure $\omega\in\mathscr{P}(\R)$ is {\em uniformly continuous across scales} if for every $\varepsilon>0$ there exists $\delta>0$ such that for any $x\in\mathbb{R}$ and $r\in(0,1]$, we have
\begin{equation}\label{eqn:ucas}
\omega(B(x,\delta r))\le\varepsilon\omega(B(x, r)).
\end{equation}
A family $\mathcal{M}$ of measures in  $\mathscr{P}(\R)$ is called {\em jointly uniformly continuous across scales} if for every $\varepsilon>0$ there exists $\delta>0$ such that (\ref{eqn:ucas}) holds for any $\omega\in \mathcal{M}$, any $x\in \R$ and any $r\in (0,1)$.

\begin{lemma}\label{lem:uc}
Under the assumption of Theorem~\ref{thm:entporous},
for any $\varepsilon>0$ there exists $\delta=\delta(\varepsilon)>0$ such that for any $\textbf{j}\in \Sigma$ and any $y\in \mathbb{R}$,
$$\pi_{\textbf{j}}\mu\,\big({B}(y,\delta)\big)<\varepsilon.$$
\end{lemma}
\begin{proof}
Arguing by contradiction, assume that this is false. Since the family of probability measures $\pi_{\textbf{j}}\mu$ is compact in the weak star topology, it follows that there exists $\textbf{j}\in \Sigma$ and $y_0\in \mathbb{R}$ such that $\pi_{\textbf{j}}\mu$ has an atom at $y_0$. This means that the set
$$X=\{x\in [0,1): W(x)=\Gamma_{\textbf{j}}(x)+y_0\}$$
has positive Lebesgue measure. Let $x_0$ be a Lebesgue density point of $X$ and let $J_n$ be the $b$-adic interval of level $n$ which contains $x_0$. Then $|J_n\cap X|/|J_n|\to 1$ as $n\to\infty$.	

Let $i_n\in \varLambda$, $n=1,2,\ldots$, be such that $b^n x_0\in [i_n/b, (i_n+1)/b)\mod 1$. Let $\textbf{i}_n=i_1i_2\cdots i_n$. Then for each $n$,
$g_{\textbf{i}_n}$ maps $[0,1)\times \R$ onto $J_n\times \R$. By Lemma~\ref{lem:TransformA},
$$\pi_{\textbf{j}}g_{\textbf{i}_n}(x,y)=\lambda^n\pi_{\textbf{j}_n} (x,y) +\pi_{\textbf{j}}g_{\textbf{i}_n}(0,0),$$
where $\textbf{j}_n=\textbf{i}_n^*\textbf{j}$. Note that $g_{\textbf{i}_n}(x, W(x))=(S_n(x), W(S_n(x))$, where
$$S_n(x)= \frac{i_n+i_{n-1}b+\cdots+i_1b^{n-1}}{b^n}+\frac{x}{b^n}.$$
Thus for $x\in S_n^{-1}(X\cap J_n)\subset [0,1)$, we have
$$W(x)-\Gamma_{\textbf{j}_n}(x)=y_n:=(y_0-\pi_{\textbf{j}}g_{\textbf{i}_n}(0,0))/\lambda^n.$$
Thus $$|\{x\in [0,1): W(x)=\Gamma_{\textbf{j}_n}(x)+y_n\}|=|X\cap J_n|/||J_n|\to 1,$$
as $n\to\infty$.
In particular, this implies that the sequence $y_n$ is bounded.
%
Let $n_k$ be a subsequence such that $\textbf{j}_{n_k}\to \textbf{j}_\infty$ and
$y_{n_k}\to y_\infty$ in respectively $\Sigma$ and $\R$. Then for Lebesgue a.e. $x\in [0,1)$, $W(x)\in \Gamma_{\textbf{j}_\infty}(x)+y_\infty$. By continuity, it follows that $W(x)=\Gamma_{\textbf{j}_\infty}(x)+y_\infty$ is a $C^1$ function, a contradiction!
\end{proof}

\begin{proposition}\label{prop:uc}
Under the assumption of Theorem~\ref{thm:entporous},
the family of measures $\{\pi_{\textbf{j}}\mu\}_{\textbf{j}\in\Sigma}$ is jointly uniformly continuous across scales.
\end{proposition}
\begin{proof}
It suffices to prove that there is $\kappa>0$ such that for any $\textbf{j}\in \Sigma$, any $x\in \R$ and any $r\in (0,1]$,
\begin{equation}\label{eqn:doubling}
\pi_{\textbf{j}}\mu(B(x,\kappa r)\le \frac{1}{2} \pi_{\textbf{j}}\mu (B(x, r)).
\end{equation}
To this end, let $\delta= \delta(1/2)>0$ be given by the previous lemma and let $M>\delta$ be a constant such that $\pi_{\textbf{j}}\mu$ is supported in $[-M, M]$ for each $\textbf{j}\in \Sigma$. Put $\kappa=\lambda\delta/(3M)$.
Given $r\in (0,1)$, choose $n=n(r)\in \mathbb{N}$ such that
$$3 M\le \lambda^{-n} r<3\lambda^{-1} M.$$
Note that $\lambda^{-n} \kappa r<\delta<M.$
We shall
show that for each $\textbf{i}=i_1i_2\cdots i_n\in \varLambda^n$,
\begin{equation}\label{eqn:compdoub}
\pi_{\textbf{j}} g_{\textbf{i}} \mu(B(x, \kappa r))\le \frac{1}{2} \pi_{\textbf{j}}g_{\textbf{i}} \mu (B(x,r)).
\end{equation}
Once this is proved, (\ref{eqn:doubling}) follows from (\ref{eqn:self`affine'}).

To prove (\ref{eqn:compdoub}), we first apply Lemma~\ref{lem:TransformA} and obtain $x(\textbf{i})\in \mathbb{R}$, such that for any $R>0$,
$$\pi_{\textbf{j}}g_{\textbf{i}}\mu (B(x,R))=\pi_{\textbf{i}^*\textbf{j}}(B(x(\textbf{i}), \lambda^{-n} R)).$$
If $|x(\textbf{i})|\ge 2M$, then $B(x(\textbf{i}), \lambda^{-n} \kappa r)$ is disjoint from $[-M, M]$ since $\lambda^{-n}\kappa r\le M$. Thus the left hand side of (\ref{eqn:compdoub}) is zero and hence the inequality holds. Assume now that $|x(\textbf{i})|< 2M$. Then
$$B(x(\textbf{i}), \lambda^{-n}r)\supset [-M, M],$$ so
the right hand side of (\ref{eqn:compdoub}) is equal to $1/2$. On the other hand,
$$B(x(\textbf{i}), \lambda^{-n}\kappa r)\subset B(x(\textbf{i}), \delta).$$
Thus the left hand side of (\ref{eqn:compdoub}) is at most $1/2$ and hence the inequality holds.
\end{proof}

\begin{corollary} $\alpha>0$.
\end{corollary}
\begin{proof} By Proposition~\ref{prop:uc}, there is $\delta>0$ such that $\pi_{\textbf{j}}\mu(B(y, \delta^n ))\le 2^{-n}\pi_{\textbf{j}}\mu (B(y,1))$ for any $\textbf{j}\in \Sigma$, $n\in \mathbb{Z}_+$ and $y\in \R$. It follows that
$$\limsup_{r\to 0} \frac{\log \pi_{\textbf{j}}\mu (B(y,r))}{\log r}\ge \log_2 \delta^{-1}>0.$$
Thus $\alpha>0$.
\end{proof}

\subsection{Entropy porosity of $\pi_{\textbf{j}}\mu$ }
In this subsection we complete the proof of Theorem~\ref{thm:entporous}.
\begin{lemma}\label{lem:boundsinm}
For any $\varepsilon>0, m\ge M(\varepsilon), n\ge N(\varepsilon,m)$,
		$$\inf\limits_{\textbf{j}\in\Sigma}\mathbb{\nu}^{n}\left(\left\{\textbf{i}\in \varLambda^{n}: \alpha-\varepsilon<\frac{1}{m}
H(\pi_{\textbf{i}\textbf{j}}\mu, \mathcal{L}_{m})<\alpha+\varepsilon\right\}\right) >1-\varepsilon.$$
\end{lemma}
\begin{proof} Denote $h_m(\textbf{j})=\frac{1}{m} H(\pi_{\textbf{j}}\mu, \mathcal{L}_m)$.
Let us first show that $h_m$ is continuous in $\textbf{j}\in\Sigma$.
Indeed, the supports of $\text{supp}(\pi_{\textbf{j}}\mu)$ are uniformly bounded and $\textbf{j}\mapsto \pi_{\textbf{j}}\mu$ is continuous in the weak star topology. Since $\pi_{\textbf{j}}\mu$ has no atom, for any $I\in \mathcal{L}_m$, $\textbf{j}\mapsto \pi_{\textbf{j}}\mu(I)$ is continuous.  Thus $$\frac{1}{m} H (\pi_{\textbf{j}}\mu, \mathcal{L}_m)=\frac1m\sum\limits_{I\in\mathcal{L}_m,\,I\subseteq[0,1]} h\big(\,\pi_{\textbf{j}}\mu(I)\big)$$ is continuous in $\textbf{j}$,
where $h(t)=t\log_b\frac1t$ is a continuous function in $[0,\infty)$.

Since $h_m$ converges to $\alpha$ $\nu^{\mathbb{Z}_+}$-a.e., the sequence $\{h_m\}$ also converges to $\alpha$ in measure, i.e.
$$\Omega_m:=\left\{\textbf{j}\in \Sigma: \left|\frac{1}{m} H(\pi_{\textbf{j}}\mu, \mathcal{L}_m)-\alpha\right|<\eps\right\}$$
satisfies $\nu^{\mathbb{Z}_+}(\Omega_m)\to 1$ as $m\to\infty$. So there exists $M(\varepsilon)$ such that when $m\ge M(\varepsilon)$, $\nu^{\mathbb{Z}_+}(\Omega_m)>1-\varepsilon/2$.

Fix such an $m\ge M(\eps)$. As $\Omega_m$ is an open subset of $\Sigma$, there exists $N:=N(m,\eps)$ such that the union $X_N$ of the $N$-th cylinders completely contained in $\Omega_m$  has $\nu^{\mathbb{Z}_+}$-measure greater than $1-\eps$. For each $n\ge N$, $X_n\supset X_N$. The lemma follows.
\end{proof}

We shall need the following two lemmas which are respectively Lemma 3.7 and Lemma 3.10 in \cite{barany2019hausdorff}.

\begin{lemma}\label{lem:entporous1}
For any $\varepsilon>0$, there exists $\delta>0$ such that the following holds.
Let $m,\ell\in\mathbb{N}$ and $k>k(m,\ell)$ be given, and suppose that $\tau\in\mathscr{P}(\mathbb{R})$ is a measure and $\beta>0$ is a constant such that for a $(1-\delta)$-fraction of $1\le t\le k$, we can write $\tau$ as a convex combination $\tau=p_0\tau_0+\sum\limits_{i\ge1}p_i\tau_i,\,\tau_i\in\mathscr{P}(\mathbb{R}),p_0<\delta$ so as to satisfy the following three conditions 
\begin{enumerate}
\item [(1)] $\frac{1}{m} H(\tau_i,\mathcal{L}_{t+m})\ge\beta,\;i\ge 1$.
\item [(2)] $\textrm{diam}(\textrm{supp}(\tau_i))\le b^{-(t+\ell)},\;i\ge 1$.
\item [(3)] $\tau(I)<\delta\tau(J)$ whenever $I\subseteq J$ are concentric intervals, $|I|=b^{-\ell}|J|=b^{-(t+\ell)}$.
\end{enumerate}
Assume further that $\left|\frac{1}{k}H(\tau,\mathcal{L}_k)-\beta\right|<\delta$.
Then $\tau$ is $(\beta,\varepsilon,m)$-entropy porous from scale $1$ to $k$.
\end{lemma}

\begin{lemma}\label{lem:entporous2}
For every $\varepsilon>0$ there exists  $\delta>0$ with the following property.
Let $\ell\in\mathbb{N}$ and $m>m(\varepsilon,\ell)$, and let $\tau\in\boldsymbol{\mathscr{P}}(\mathbb{R})$ be a measure such that $\tau(I)<\frac{\delta}2\tau(J)$ whenever $I\subseteq J$ are concentric intervals, $|I|=b^{-\ell}|J|=2b^{-(k+\ell)}$ for every $k\in\mathbb{N}$. Let $n>n(m,\ell)$ and suppose that $\tau$ is $(\alpha,\delta,m)$-entropy porous from scales $n_1\,to\,n_2=n_1+n$. Then
for any $f(x)=ax+c$, $a\in \R\setminus \{0\}$ and $c\in\mathbb{R}$, $f\tau$ is $(\alpha,\varepsilon,m)$-entropy porous from scales $n_1-[\log_b |a|]$ to $n_2-[\log_b |a|]$.
\end{lemma}

\begin{lemma}\label{lem:entporous3}
Under the assumption of Theorem~\ref{thm:entporous}, for any $\varepsilon>0$, there exists $\delta>0$ such that if 	
$m\ge M(\varepsilon)$ and $k\ge K(\varepsilon,m)$ and if $\left|\frac{1}{k} H(\pi_{\textbf{j}}\mu, \mathcal{L}_k)-\alpha\right|<\frac{\delta}{2}$, then $\pi_{\textbf{j}}\mu$ is $(\alpha, \varepsilon, m)$-entropy porous from scale $1$ to $k$.
\end{lemma}
\begin{proof}
(1) Assume without loss of generality that $\pi_{\textbf{j}}\mu$ is supported in $[0,1]$ for all $\textbf{j}\in \Sigma$.
Fix $\eps>0$. Let $\delta>0$ be so small that the conclusion of Lemma~\ref{lem:entporous1} holds and $\delta<2\alpha$. Let $\beta=\alpha-\delta/2>0.$
By Proposition~\ref{prop:uc}, there exists $\ell\in \mathbb{N}$, such that for any  $\textbf{j}\in\Sigma$, we have
\begin{equation}\label{eqn:ucc}
\pi_{\textbf{j}}\mu(I)<\frac{\delta}2\pi_{\textbf{j}}\mu(J)
\end{equation}
whenever $I\subseteq J$ are concentric intervals with $1\ge |I|=b^{-\ell}|J|$.
	
By Lemma~\ref{lem:boundsinm}, when $m\ge M(\eps)$ and $n\ge N(\eps, m)$, we have 
 \begin{equation}
 \label{3.9.1}
 \nu^{\hat{n}}\left(\left\{\textbf{i}\in \varLambda^{\hat{n}}: \alpha-\frac{\delta} 6 < \frac{1}{m}H(\pi_{\textbf{i}\textbf{j}}\mu, \mathcal{L}_{m})< \alpha+\frac{\delta}6 \right\}\right) >1-\delta
 \end{equation}
Increasing $M(\eps)$ if necessary, we may assume that $M(\eps)>6\max(C_0, \ell)/\delta$, where $C_0$ is as in (\ref{eqn:jiij}).

Fix $m>M(\eps)$ and assume $k>K(\eps, m):=N(\eps, m)/\delta$.  Let us show that for any $N(\eps,m)<n\le k$, and for $t=n-\ell$, the measure $\tau=\pi_{\textbf{j}}\mu$ can be written in the form $\sum p_i \tau_i$ with the properties (1)-(3) in Lemma~\ref{lem:entporous1}.

Indeed, since $m> 6C_0/\delta$, by (\ref{eqn:jiij}), for any $\textbf{i}\in \varLambda^{\hat{n}}$,
$$\left|\frac{1}{m} H(\pi_{\textbf{j}}g_{\textbf{i}}\mu, \mathcal{L}_{n+m})-\frac{1}{m} H(\pi_{{\textbf{i}}^*\textbf{j}}\mu, \mathcal{L}_m)\right|
\le \frac{C_0}{m}<\frac{\delta}{6}.$$
So (\ref{3.9.1}) implies that the set
$$\mathcal{I}_n=\left\{\textbf{i}\in \varLambda^{\hat{n}}:
\frac{1}{m}H(\pi_{\textbf{j}}g_{\textbf{i}}\mu, \mathcal{L}_{m+n})>\alpha-\frac{\delta}{3}\right\}$$
has cardinality greater than $(1-\delta) b^{\hat{n}}$. We define $\tau_1,\tau_2,\ldots$ to be equal to $\pi_{\textbf{j}}g_{\textbf{i}}\mu$ with $\textbf{i}\in \mathcal{I}_n$, $p_1=p_2=\cdots= b^{-\hat{n}}$ and define $p_0=1-\#\mathcal{I}_n b^{-\hat{n}}$ and
$\tau_0$ to be the average of $\pi_{\textbf{j}}g_{\textbf{i}}\mu$ for those $\textbf{i}\in \varLambda^{\hat{n}}\setminus \mathcal{I}_n$. Then $\tau=p_0\tau_0+p_1\tau_1+\cdots $ and $p_0<\delta$. Moreover,
\begin{enumerate}
\item [(1)] For each $i=1,2,\ldots$,
$$\frac{1}{m}H(\tau_i,\mathcal{L}_{t+m})\ge \frac{1}{m} \left(H(\tau_i, \mathcal{L}_{n+m})-\ell \right)> \alpha-\frac{\delta}{2}=\beta.$$
\item [(2)] Since we assume that all the $\pi_{\textbf{j}}\mu$ are supported in $[0,1]$ and $\lambda^{\hat{n}}\le b^{-n}$ by definition of $\hat{n}$, by Lemma~\ref{lem:TransformA}, each of $\tau_1,\tau_2, \cdots$ is supported in an interval of length $b^{-n}\le b^{-(t+\ell)}$.
\item [(3)] The property (3) follows from (\ref{eqn:ucc}).
\end{enumerate}
Since
$$\left|\frac{1}{k} H(\tau, \mathcal{L}_k)-\beta\right|\le \left|\frac{1}{k} H(\tau, \mathcal{L}_k)-\alpha\right|+|\alpha-\beta|<\delta,$$
by Lemma~\ref{lem:entporous1}, we obtain that $\tau$ is $(\beta, \eps, m)$-entropy porous from scale $1$ to $k$, hence it is $(\alpha, \eps,m)$- entropy porous from scale $1$ to $k$.
\end{proof}

\begin{proof} [Proof of Theorem~\ref{thm:entporous}]
By Lemma~\ref{lem:TransformA}, $$\pi_{\textbf{j}}g_{\textbf{w}} g_{i_1i_2\cdots i_{\hat{n}}}=\lambda^{\hat{n}+|w|} \pi_{i_{\hat{n}}\cdots {i_1}{\textbf{w}}^*\textbf{j}}\mu+\text{Constant}.$$
So by Lemma~\ref{lem:entporous2}, it suffices to prove that when $m>M(\eps)$, $k\ge K(\eps, m)$ and $n\ge N(\eps, m,k)$,  for any $\textbf{h}\in\Sigma$,
\begin{equation}\label{eqn:entpoH}
\nu\left(\left\{\textbf{i}\in \Sigma: \pi_{i_1i_2\cdots i_{\hat{n}}\textbf{h}}\mu \text{ is } (\alpha, \eps, m)-\text{entropy porous from scale } 1 \text{ to } k\right\}\right)>1-\eps.
\end{equation}
Given $\eps>0$, let $\delta$, $M(\eps)$ and $K(\eps, m)$ be given by Lemma~\ref{lem:entporous3}. For this $\delta>0$, by  Lemma~\ref{lem:boundsinm},  when $k\ge K(\delta)$ and $n\ge N(\delta, k)$, 
$$\nu\left(\left\{\textbf{i}\in \Sigma: \left|\frac{1}{k} H(\pi_{i_{\hat{n}}i_{\hat{n}-1}\cdots i_1\textbf{h}}\mu,\mathcal{L}_k)-\alpha \right|<\frac{\delta}{2}\right\}\right)>1-\delta.$$
Therefore, when $m\ge M(\eps)$, $k\ge \max(K(\eps, m), K(\delta))$ and $n\ge N(\delta,k)$, (\ref{eqn:entpoH}) holds.
\end{proof}

%
%
\section{Transversality}\label{sec:separation}
In this section, we deduce from the condition (H) some quantified estimates. These estimates will be used to construct a sequence of partitions $\mathcal{L}_n^\#$ of the space $\mathcal{X}$ in the next section which in turn is used in the last section to prove Theorem B.
The main result of this section is summarized in the following theorem.

\begin{theorem}\label{thm:estimate}
Suppose that a real analytic $\mathbb{Z}$-periodic function $\phi(x)$ satisfies the condition (H) for some integer $b\ge 2$ and $\lambda\in (1/b,1)$. Then there exist positive integers $\ell_0, Q_0$ and a constant $\rho_0>0$ with the following property.
For any $\textbf{u},\,\textbf{v}\in\Sigma$ with $u_n\ne v_n$,
\begin{equation}\label{eqn:Gamma1stder}
\sup_{x\in [0,1)} |\Gamma'_{\textbf{u}}(x)-\Gamma'_{\textbf{v}}(x)|\ge \rho_0 b^{-Q_0n},
\end{equation}
and
\begin{equation}
\sum_{\substack{I\in \mathcal{L}_{\ell_0}\\ I\subset [0,1)}}
\inf_{x\in I} |\Gamma'_{\textbf{u}}(x)-\Gamma'_{\textbf{v}}(x)|\ge \rho_0 \sup_{x\in [0,1]} |\Gamma'_{\textbf{u}}(x)-\Gamma'_{\textbf{v}}(x)|.
\end{equation}
\end{theorem}

For the proof, we observe that for any integer $k\ge 0$, the family $\Gamma_{\textbf{u}}^{(k)}$, $\textbf{u}\in \Sigma$, is compact with respect to the topology of uniform convergence in $\R$. Together with the condition (H), this implies the maps in
\begin{equation}\label{eqn:funclassFn}
\mathcal{F}_n:=\{\Gamma_{\textbf{u}}-\Gamma_{\textbf{v}}: u_n\not=v_n \text{ and } u_j=v_j \text{ for } 1\le j<n\}
\end{equation}
are uniformly separated with constants depending on $n$. In order to quantify the dependence of the constants in $n$, we shall use
%
the following fact frequently, which can be checked directly by definition of $\Gamma$:

{\em If $\textbf{u}=(u_m)_{m=1}^\infty, \textbf{v}=(v_m)_{m=1}^\infty\in \Sigma$ and $u_1=v_1$, $u_2=v_2$, $\ldots$, $u_{n-1}=v_{n-1}$ but $u_n\not=v_n$, where $n\in \mathbb{Z}_+$, then for any $k\ge 1$,
\begin{equation}\label{eqn:unnot=vn}
\Gamma_{\textbf{u},\textbf{v}}^{(k)}(x)
=\left(\frac{\gamma}{b^{k-1}}\right)^{n-1}\Gamma^{(k)}_{\sigma^{n-1}(\textbf{u}),\sigma^{n-1}(\textbf{v})}
\left(\frac{x+u_1+\cdots+u_{n-1}b^{n-2}}{b^{n-1}}\right),
\end{equation}
where $\Gamma_{\textbf{u},\textbf{v}}=\Gamma_{\textbf{u}}-\Gamma_{\textbf{v}}$.}

%
%

\begin{definition} For an integer $k\ge 0$, we say that a map $\psi: [a,b)\to \R$ is {\em $k$-regular} if $\psi$ is $C^k$ and
$$\sup_{x\in [a,b)} |\psi^{(k)}(x)|\le  2\inf_{x\in [a,b)} |\psi^{(k)}(x)|.$$
\end{definition}

\begin{lemma}\label{lem:derbound}
There exists a constant $\eps_1>0$ and a positive integer $Q_1$ such that for any $f\in \mathcal{F}_1$,
\begin{equation}\label{eqn:C1low}
\sup_{x\in [0,1]} |f'(x)|\ge \eps_1,
\end{equation}
and for any $x\in [0,1]$, there exists $k\in \{1,2,\ldots, Q_1\}$ such that
\begin{equation}\label{eqn:someder}
|f^{(k)}(x)|\ge \eps_1.
\end{equation}
\end{lemma}
\begin{proof} This follows from the fact that $\mathcal{F}_1$ is compact with respect to the topology of uniform convergence in the $C^k$ sense.
 More precisely, if (\ref{eqn:C1low}) fails, then there exists $f_m\in \mathcal{F}_1$ such that $\sup_{x\in [0,1]} |f_m'(x)|<1/m$. Passing to a subsequence we may assume that there exists $f\in \mathcal{F}_1$ such that $\sup_{x\in \R} |f'_m(x)-f'(x)|\to 0$. Then $f'(x)=0$ for all $x\in [0,1]$. Since $f$ is real analytic and $f(0)=0$, this implies that $f\equiv 0$. However, $\mathcal{F}_1$ does not contain the zero function by the condition (H), a contradiction.

Similarly, if (\ref{eqn:someder}) fails, then there exists $f_m\in \mathcal{F}_1$ and $x_m\in [0,1]$ such that $|f_m^{(k)}(x_m)|<1/m$, for each $m=1,2,\ldots$ and $k=1,2,\ldots, m$. Passing to a subsequence, there exists $x_0\in [0,1]$ and $f\in \mathcal{F}_1$ such that $x_m\to x_0$ and $\max_{x\in [0,1]}|f_m^{(k)}(x)-f^{(k)}(x)|\to 0$ as $m\to\infty$, for each $k=1,2,\ldots$. It follows that $f^{(k)}(x_0)=0$ for all $k\ge 1$. As $f$ is real analytic and $f(0)=0$, this implies that $f\equiv 0$, a contradiction.
\end{proof}

\begin{lemma}\label{lem:decombd}
Let $\eps_1, Q_1$ be as in Lemma~\ref{lem:derbound}. There exist $\ell_1\in \mathbb{N}$ such that for any $f\in \mathcal{F}_n$, $n=1,2,\ldots$, and any $I\in \mathcal{L}_{\ell_1}$ with $I\subset [0,1)$, $f: I\to \mathbb{R}$ is $k$-regular and
$$\sup_{x\in I} |f^{(k)}(x)|\ge \eps_1\left(\gamma b^{1-k}\right)^{n-1}$$
for some $k\in \{1,2,\ldots, Q_1\}$.
\end{lemma}
\begin{proof}
For $n=1$, there is a constant $C_1>0$ such that for each $f\in \mathcal{F}_1$, $|f^{(k)}(x)|\le C_1$ for any $k=2,3,\ldots, Q_1+1$ and any $x\in [0,1)$. So there is an integer $\ell_1\ge 1$ such that for each $f\in \mathcal{F}_1$ and each $I\in \mathcal{L}_{\ell_1}$ with $I\subset [0,1)$, $f:I\to\R$ is $k$-regular and $\sup_{x\in I} |f^{(k)}(x)|\ge \eps_1$ for some $k\in \{1,2,\ldots, Q_1\}$.

For general $n$, this follows from (\ref{eqn:unnot=vn}). Indeed,
there is $u_1, u_2,\ldots, u_{n-1}$ and a map $f_1\in \mathcal{F}_1$ such that
$$f'(x)= \gamma^{n-1} f'_1\left(\frac{x+u_1+\cdots+u_{n-1}b^{n-2}}{b^{n-1}}\right).$$
For any $I\in \mathcal{L}_{\ell_1}$, there is $J\in \mathcal{L}_{\ell_1}$ such that
$$x\in I\Rightarrow \frac{x+u_1+\cdots+u_{n-1}b^{n-2}}{b^{n-1}}\in J.$$
Since $f_1$ is $k$-regular in $J$ for some $k\in \{1,2,\ldots, Q_1\}$, $f$ is $k$-regular in $I$ for the same $k$.
\end{proof}

\begin{lemma}\label{lem:valuebd0}
For any integer $k\ge 1$, there exist $\delta_k>0$ and $\tau_k>0$ such that the following holds. Let $\psi:[0,1)\to \R$ be a $C^k$ function such that $|\psi^{(k)}(x)|\ge 1$ for all $x\in [0,1)$. Then there exists a subinterval $J$ of $[0,1)$ such that $|J|\ge \delta_k$ and such that $|\psi'(x)|>\tau_k$ for all $x\in J$.
\end{lemma}

\begin{proof} We prove by induction on $k$. The starting step $k=1$ is trivial. Now assume that the lemma is true for $k<m$, $m\ge 2$. Let us prove it for the case $k=m$. Assume without loss of generality that $\psi^{(m)}(x)\ge 1$ for all $x\in [0,1)$ and that $\psi^{(m-1)}(1/2)\ge 0$. Then $\psi^{(m-1)}(x)\ge \frac{1}{4}$ for all $x\in [3/4, 1)$. Consider the function $\varphi(x)=4^m\psi((x+3)/4)$. Then $\varphi^{(m-1)}(x)\ge 1$ for all $x\in [0,1)$. By the induction hypothesis, there is a subinterval $J_{m-1}$ of $[0,1)$ such that $|J_{m-1}|\ge \delta_{m-1}$ and $|\varphi'(x)|\ge \tau_{m-1}$ for all $x\in J_{m-1}$. Put $J_m=\{(x+3)/4: x\in J_{m-1}\}$, $\delta_m=\delta_{m-1}/4$ and $\tau_m=(1/4)^{m-1}\tau_{m-1}$. Then $|J_m|\ge \delta_m$ and $|\psi'(x)|\ge \tau_m$ for all $x\in J_m$.
\end{proof}

\begin{lemma}\label{lem:bdlocal2global} Assume that $f: [a,b)\to \R$ is $k$-regular for some positive integer $k$.
	Then there exists $\delta_k>0$, $\rho_k>0$ depending only on $k$ and an interval $J\subset [a,b)$ with $|J|>\delta_k (b-a)$ such that
$$\inf_{x\in J}|f'(x)|\ge \rho_k \sup_{x\in [a,b)} |f'(x)|.$$
%
\end{lemma}
\begin{proof}
Without loss of generality, we may assume that $[a,b)=[0,1)$ and $$\sup_{0\le x<1} |f^{(k)}(x)|=1.$$
(Otherwise, we consider $\lambda_1 f(\lambda_2 x+c)$ instead of $f$ for suitable choices of $\lambda_1, \lambda_2>0$ and $c\in \R$.) We may assume that for each $1\le k'<k$, $f:[0,1)\to \R$ is not $k'$-regular, i.e.
	\begin{equation}\label{eqn:smallordder}%
	\sup_{x\in [0,1)}|f^{(k')}(x)|> 2 \inf_{x\in [0,1)} |f^{(k')}(x)|,
	\end{equation}
for otherwise we may work on $k'$ instead of $k$.
By the mean value theorem,
	$$|f^{(k-1)}(x)-f^{(k-1)}(y)|\le |x-y|\le 1$$
	for each $x,y\in [0,1)$. By (\ref{eqn:smallordder}), we have $\sup_{x\in [0,1)} |f^{(k-1)}(x)|<2.$
	But then by the mean value theorem again
	$$|f^{(k-2)}(x)-f^{(k-2)}(y)|\le 2 |x-y|\le 2$$
	holds for all $x,y\in [0,1]$. Once again by (\ref{eqn:smallordder}), we obtain
	$\sup_{x\in [0,1)} |f^{(k-2)}(x)|\le 4.$
	Repeating the process, 
	$$\sup_{x,y\in [0,1)}|f'(x)-f'(y)|\le 2^{k-1}.$$
    On the other hand, by Lemma~\ref{lem:valuebd0}, there exists $\delta_k>0$, $\tau_k>0$ and 	an interval $J$ with $|J|\ge \delta_k$ such that  $|f'(x)|>\tau_k$ for all $x\in J$. The lemma follows by taking $\rho_k=\tau_k/(2^{k-1}+\tau_k)$.
\end{proof}

\begin{proof}[Proof of Theorem~\ref{thm:estimate}]
By Lemma~\ref{lem:decombd} and Lemma~\ref{lem:valuebd0}, we obtain the first inequality. By Lemma~\ref{lem:decombd} and Lemma~\ref{lem:bdlocal2global}, we obtain the second inequality.
\end{proof}

\section{The partitions $\mathcal{L}_i^{\mathcal{X}}$ of the space $\mathcal{X}$}\label{sec:partitionX}
In this section, we construct a nested sequence of partitions $\mathcal{L}_i^{\mathcal{X}}$ of the space $\mathcal{X}$ in (\ref{FunctionSet}) and prove a few key properties of these partitions. The separation properties given in Theorem~\ref{thm:estimate} play a central role in the proofs.

Throughout we fix an integer $b\ge 2$ and $\lambda\in (1/b, 1)$ and we assume that $\phi:\R\to \R$ is a real analytic $\Z$-periodic function that satisfies the condition (H).

Recall that by Lemma~\ref{lem:TransformA},
for any $\textbf{j}\in\Sigma$ and $\textbf{i}\in\varLambda^\#$,
$$\pi_{\textbf{j}}g_{\textbf{i}}(x,y)
=\lambda^{|\textbf{i}|}(y-\Gamma_{{\textbf{i}}^*\textbf{j}}(x))+ \pi_{\textbf{j}}g_{\textbf{i}}(0).$$
So each member of $\mathcal{X}$ can be written in the form $\lambda^{t}(y-\psi(x))+c,$ where $t\in\mathbb{N},c\in\mathbb{R}$ and $\psi(x)\in C^\omega(\R)$ with $\psi(0)=0$. We shall call $|\textbf{i}|$ the {\em height} of the map $\pi_{\textbf{j}}g_{\textbf{i}}$.
Define $\overline{\pi}:\mathcal{X}\to\mathbb{N}\times\mathbb{R}^{M+1}$ by
$$\lambda^{t}(y+\psi(x))+c\to\bigg(t,\,\psi(\frac1M),\,\psi(\frac2M),\ldots,\;\psi(1),\;c\bigg),$$
where $M=b^{\ell_0}$ and $\ell_0$  comes from Theorem~\ref{thm:estimate}.

\begin{definition} For each integer $n\ge 1$,
$\mathcal{L}_n^{\mathcal{X}}$ consists of non-empty subsets of $\mathcal{X}$ of the following form
$$\overline{\pi}^{-1}\left(\{t\}\times I_1 \times I_2 \times\ldots\times I_M \times J\right),$$
where $t\in \mathbb{N}, I_1,I_2,\cdots, I_M\in \mathcal{L}_n, J\in \mathcal{L}_{n+[t\log_b 1/\lambda]}.$
The partition $\mathcal{L}_0^{\mathcal{X}}$ consists of non-empty subsets of $\mathcal{X}$ of the following form
$$\overline{\pi}^{-1}\left(\{t\}\times\mathbb{R}\times\ldots\times\mathbb{R}\times J\right),$$
where $t\in \mathbb{N}, J\in \mathcal{L}_{[t\log_b 1/\lambda]}.$
\end{definition}

\begin{lemma}\label{lem:boundspart}
There exists $A>0$ such that any $i\ge 0$, each element of $\mathcal{L}_i^{\mathcal{X}}$ contains at most $A$ elements of $\mathcal{L}_{i+1}^{\mathcal{X}}$.
\end{lemma}
\begin{proof} When $i\ge 1$, the statement holds with $A=b^{M+1}$. Since $\Gamma_{\textbf{j}}(x)$ is uniformly bounded in $[0,1)$, $\textbf{j}\in\Sigma$, for each $t\in \mathbb{N}$, there are only finitely many members of $\mathcal{L}_1^{\mathcal{X}}$ whose elements have height $t$. So enlarging $A$, we can guarantee that the statement holds also for the case $i=0$.
\end{proof}
\begin{lemma}\label{lem:constantR}
There exists $R>0$ such that if $\pi_{\textbf{j}}g_{\textbf{u}}$ and $\pi_{\textbf{j}}g_{\textbf{v}}$ belong to the same element of $\mathcal{L}_i^{\mathcal{X}}$, where $\textbf{j}\in \Sigma$, $\textbf{u}, \textbf{v}\in \varLambda^{\hat{n}}$, and $i\ge 1$, then for any $x\in [0,1)$ and $y\in \R$,
$$|\pi_{\textbf{j}}g_{\textbf{u}}(x,y)-\pi_{\textbf{j}}g_{\textbf{v}}(x,y)|\le R b^{-(n+i)}.$$
\end{lemma}
\begin{proof} By definition of the partition $\mathcal{L}_i^{\mathcal{X}}$, we have $$|\pi_{\textbf{j}}g_{\textbf{u}}(\textbf{0})-\pi_{\textbf{j}}g_{\textbf{v}}(\textbf{0})|=O(b^{-(n+i)})$$ and
$$|\Gamma_{{\textbf{u}}^*\textbf{j}}(k/M)-\Gamma_{{\textbf{v}}^*\textbf{j}}(k/M)|\le b^{-i}$$
for each $1\le k\le M$. Note that the last inequality also holds for $k=0$ since then the left hand side is equal to $0$.
For each $I\in \mathcal{L}_{\ell_0}$ with $I\subset [0,1)$ there exists $0\le k<M$ such that $I=[k/M, (k+1)/M)$. Thus
$$\inf_{x\in I}|\Gamma'_{{\textbf{u}}^*\textbf{j}}(x)-\Gamma'_{{\textbf{v}}^*\textbf{j}}(x)|\le 2b^{-i}M.$$
By Theorem~\ref{thm:estimate}, it follows that
\begin{equation}\label{eqn:Gamma'diff}
\sup_{x\in [0,1)}|\Gamma'_{{\textbf{u}}^*\textbf{j}}(x)-\Gamma'_{{\textbf{v}}^*\textbf{j}}(x)|\le 2\rho_0^{-1}M b^{-i},
\end{equation}
hence
\begin{equation}\label{eqn:Gammadiff}
\sup_{x\in [0,1)}|\Gamma_{{\textbf{u}}^*\textbf{j}}(x)-\Gamma_{{\textbf{v}}^*\textbf{j}}(x)|\le 2\rho_0^{-1}M b^{-i}.
\end{equation}
Since
$$\pi_{\textbf{j}}g_{\textbf{u}}(x,y)-\pi_{\textbf{j}}g_{\textbf{v}}(x,y)=-\lambda^{\hat{n}} (\Gamma_{{\textbf{u}}^*\textbf{j}}(x)-\Gamma_{{\textbf{v}}^*\textbf{j}}(x))+
\pi_{\textbf{j}}g_{\textbf{u}}(\textbf{0})-\pi_{\textbf{j}}g_{\textbf{v}}(\textbf{0})$$
the lemma follows.
\end{proof}

\begin{lemma}\label{lem:constantC}
There exists a constant $C\in \mathbb{Z}_+$ 	such that for any  $\textbf{u}\ne\textbf{v}\in\varLambda^n$, $n\ge 1$, and $\textbf{j}\in\Sigma$,
$\mathcal{L}_{Cn}^{\mathcal{X}}(\pi_{\textbf{j}}g_{\textbf{u}})\ne\mathcal{L}_{Cn}^{\mathcal{X}}(\pi_{\textbf{j}}g_{\textbf{v}})$.
\end{lemma}
\begin{proof} Choose $C\in \mathbb{Z}_+$ such that $$\rho_0 b^{-Q_0n}> 2\rho_0^{-1}M b^{-Cn}$$
holds for all $n\ge 1$.
Since $\textbf{u}$ and $\textbf{v}$ are distinct elements of $\varLambda^n$, by Theorem~\ref{thm:estimate},
$$\sup_{x\in [0,1)}|\Gamma'_{{\textbf{u}}^*\textbf{j}}(x)-\Gamma'_{{\textbf{v}}^*\textbf{j}}(x)|\ge \rho_0 b^{-Q_0 n}> 2\rho_0^{-1}M b^{-Cn}.$$
As in the proof of (\ref{eqn:Gamma'diff}), we see that $\pi_{\textbf{j}}g_{\textbf{u}}$ and $\pi_{\textbf{j}}g_{\textbf{v}}$ cannot belong to the same element of $\mathcal{L}_{Cn}^{\mathcal{X}}$.
\end{proof}

For a discrete probability measure $\eta$ in the space $\mathcal{X}$ and a Borel probability measure $\mu$ in $\mathbb{R}$, let $\eta.\mu$ denote the Borel probability  measure in $\mathbb{R}$ such that for any Borel subset of $\mathbb{R}$,  
$$\eta.\mu(A)=\eta\times\mu(\{(\Psi, x)\in\mathcal{X}\times \mathbb{R}: \Psi(x)\in A\}).$$
\begin{lemma}\label{lem:entinceta} For any $\eps>0$, there exists $p>0$ and $\delta_*>0$ such that the following holds if $i$ and $k$ are sufficiently large. If $\eta$ is a probability measure supported in an element of $\mathcal{L}_i^{\mathcal{X}}$ such that each element in the support of $\eta$ has height $\hat{n}$ and such that
$$\frac{1}{k}H\big(\eta,\mathcal{L}_{i+k}^{\mathcal{X}}\big)>\eps,$$
then
$$\nu^{\hat{i}}\left(\left\{\textbf{u}\in \Sigma^{\hat{i}}: \frac{1}{k}H\big(\eta.\big(\delta_{g_{\textbf{u}}(\textbf{0})}\big),\mathcal{L}_{i+k+n}\big)\ge \delta_*\right\}\right)>p.$$
\end{lemma}

\begin{proof}
Let $M_1=b^{\ell_0+1}$, where $\ell_0$ is as in Theorem~\ref{thm:estimate} and assume $\hat{i}>\ell_0$.  It suffices to prove that for each integer $0\le T<b^{\hat{i}-\ell_0-1}$, there exists at least one element $x$ of
$$\mathcal{X}_T=\left\{\frac{T}{b^{\hat{i}}} +\frac{j}{M_1}: 0\le j<M_1, j\in \mathbb{Z}\right\}$$
such that
\begin{equation}\label{eqn:entxi}
\frac{1}{k}H(\eta.\delta_{(x,W(x))}, \mathcal{L}_{i+k+n})>\frac{\eps}{2M_1},
\end{equation}
Indeed, once this proved, the desired estimate holds with $\delta_*=\eps/(2M_1)$ and $p=1/M_1$.

So let us fix $T$.  Write $\widetilde{x}_j=\frac{T}{b^{\hat{i}}}+\frac{j}{M_1}$, $0\le j<M_1$ and let $\widetilde{z}_j=\big(\widetilde{x}_j,W(\widetilde{x}_j)\big)$.
Define $F: \text{supp} (\eta)\to\mathbb{R}^{M_1}$, by
$$F(\Psi)=\big(\Psi(\widetilde{z}_0),\Psi(\widetilde{z}_1),\ldots,\Psi(\widetilde{z}_{M_1-1})\big).$$

{\bf Claim.} There exists a constant $\widetilde{C}$ such that
$$H\big(\eta,\mathcal{L}_{i+k}^{\mathcal{X}}\big)
\le H\big( F\eta,\mathcal{L}_{i+k+n}^{\mathbb{R}^{M_1}}\big)+\widetilde{C}.$$

To prove this claim, take $I\in\mathcal{L}_{i+k+n}^{\mathbb{R}^{M_1}}$.
It suffices to show that the cardinality of the set $\{J\in\mathcal{L}_{i+k}^{\mathcal{X}}\big| J\cap F^{-1}(I)\ne\emptyset\mbox{ and }J\cap \text{supp}(\eta)\ne\emptyset\}$ is uniformly bounded.
For any $\Psi^{(m)}\in \text{supp}(\eta)$ with $F(\Psi^{(m)})\in I$, $m=1,2$, write $\Psi^{(m)}(x,y)=\lambda^{\hat{n}}(y-\Gamma_{\textbf{u}^{(m)}}(x))+c^{(m)}$.
For each $1\le j<M_1$, 
$$\left|\left(\Psi^{(2)}(\widetilde{z}_j)-\Psi^{(1)}(\widetilde{z}_j)\right)-\left(\Psi^{(2)}(\widetilde{z}_{j-1})-\Psi^{(1)}(\widetilde{z}_{j-1})\right)
\right|=O(b^{-(i+k+n)})$$
which means that
$$\lambda^{\hat{n}}\left|\left(\Gamma_{\textbf{u}^{(2)}}-\Gamma_{\textbf{u}^{(1)}}\right)(\widetilde{x}_{j})
-\left(\Gamma_{\textbf{u}^{(2)}}-\Gamma_{\textbf{u}^{(1)}}\right)(\widetilde{x}_{j-1})\right|=O(b^{-(i+k+n)}),$$
i.e.
$$\left|\left(\Gamma_{\textbf{u}^{(2)}}-\Gamma_{\textbf{u}^{(1)}}\right)(\widetilde{x}_{j})-
\left(\Gamma_{\textbf{u}^{(2)}}-\Gamma_{\textbf{u}^{(1)}}\right)(\widetilde{x}_{j-1})\right|=O(b^{-(i+k)}).$$
Therefore,
$$\inf_{x\in [\widetilde{x}_{j-1}, \widetilde{x}_j)} |\Gamma'_{\textbf{u}^{(2)}}(x)-\Gamma'_{\textbf{u}^{(1)}}(x)|= O(b^{-(i+k)}).$$
For each element $L$ of $\mathcal{L}_{\ell_0}$ which is contained in $[0,1)$
there exists $1\le j<M_1$ such that $[\widetilde{x}_{j-1}, \widetilde{x}_j)\subset L$.
So $$\inf_{x\in L }|\Gamma'_{\textbf{u}^{(2)}}(x)-\Gamma'_{\textbf{u}^{(1)}}(x)|= O(b^{-(i+k)}).$$
By Theorem~\ref{thm:estimate}, it follows that  $$\sup_{x\in[0,1]}\bigg|\big(\Gamma_{\textbf{u}^{(2)}}-\Gamma_{\textbf{u}^{(1)}}\big)'(x)\bigg|={O}(b^{-(i+k)}).$$
Since $\Gamma_{\textbf{j}}(0)=0$ for each $\textbf{j}\in \Sigma$, we obtain that
$$\sup_{x\in [0,1)} |\Gamma_{\textbf{u}^{(2)}}(x)-\Gamma_{\textbf{u}^{(1)}}(x)|= O(b^{-(i+k)}).$$
In particular, $\lambda^{\hat{n}}\left|\left(\Gamma_{\textbf{u}^{(2)}}-\Gamma_{\textbf{u}^{(1)}}\right)(\widetilde{x}_{j})\right|=O(b^{-(i+k+n)}).$
Since $$\Psi^{(2)}(\widetilde{z}_j)-\Psi^{(1)}(\widetilde{z}_j)=-\lambda^{\hat{n}} \left(\Gamma_{\textbf{u}^{(2)}}-\Gamma_{\textbf{u}^{(1)}}\right)(\widetilde{x}_{j})+ c^{(2)}-c^{(1)},$$ we also obtain that
 \begin{equation*}
 \big|c^{(2)}-c^{(1)}\big|=O(b^{-(i+k+n)}).
 \end{equation*}
By definition of $\mathcal{L}_{i+k}^{\mathcal{X}}$, we conclude the proof of the claim.

Since
$$H\big(\ F\eta,\mathcal{L}_{i+k+n}^{\mathbb{R}^{M_1}}\big)\le\sum\limits_{j=0}^{M_1-1}
H\big(\ \eta.\delta_{\widetilde{z}_j},\mathcal{L}_{i+k+n}\big),$$	
the claim implies that for at least one $\widetilde{z}_j$ we have
$$\frac{1}{k} H\big(\eta.\delta_{\widetilde{z}_j},\mathcal{L}_{i+k+n}\big)\ge\frac{\eps}{M_1}-\frac{\widetilde{C}}{kM_1}.$$
So (\ref{eqn:entxi}) follows provided that $k$ is sufficiently large.
\end{proof}

\section{Proof of Theorem B'}\label{sec:pfThmB'}
In this section, we shall apply Hochman's criterion on entropy increasing to complete the proof of Theorem B'. The basic idea is to introduce a discrete measure $$\theta^{\textbf{j}}_n=\frac{1}{b^{\hat{n}}}\sum_{\textbf{i}\in \varLambda^{\hat{n}}} \delta_{\pi_{\textbf{j}}g_{\textbf{i}}}
\in\boldsymbol{\mathscr{P}}(\mathcal{X})$$
for each $n\in \mathbb{Z}_+$ and analyze the entropy of $\theta^{\textbf{j}}_n$ with respect to the partitions $\mathcal{L}_i^{\mathcal{X}}$ and also the entropy of $$\pi_{\textbf{j}}\mu=\theta_n^{\textbf{j}}.\mu$$
with respect to the partitions $\mathcal{L}_i$.

\subsection{The entropy of $\theta_n^{\textbf{j}}$}
We start with analyzing the entropy of $\theta_n^{\textbf{j}}$ with respect to the partitions $\mathcal{L}_i^{\mathcal{X}}$.

\begin{lemma}\label{lem:thetanL0} For $\nu^{\mathbb{Z}_+}$-a.e. $\textbf{j}\in \Sigma$, $$\lim_{n\to\infty}\frac{1}{n}{H}\left(\theta^{\,\textbf{j}}_n,\mathcal{L}_0^{\mathcal{X}}\right)=\lim_{n\to\infty} \frac{1}{n} H(\pi_{\textbf{j}}\mu, \mathcal{L}_n)=\alpha.$$
\end{lemma}
\begin{proof}	
Define $\pi_n, \pi: \Sigma\;\to\mathbb{R}^2$, by $\pi_n(\textbf{i})= g_{i_1i_2\cdots i_{\hat{n}}}(0,0)$ and $\pi(\textbf{i})=\lim_{n\to\infty} \pi_n(\textbf{i})$.
Then $\pi_n-\pi=O(b^{-n})$, and hence $\pi_{\textbf{j}}\pi_n-\pi_{\textbf{j}}\pi=O(b^{-n})$.
Therefore,
$$H(\pi_{\textbf{j}}\mu, \mathcal{L}_n)=H\left(\pi_{\textbf{j}}\pi \nu^{\mathbb{Z}_+}, \mathcal{L}_n \right)=H(\pi_{\textbf{j}}\pi_n \nu^{\mathbb{Z}_+}, \mathcal{L}_n )+O(1).$$	
For $\nu$-a.e. $\textbf{j}\in \Sigma$, $\lim_{n\to\infty} \frac{1}{n} H(\pi_{\textbf{j}}\mu,\mathcal{L}_n)=\alpha$, so
$$\lim_{n\to\infty} \frac{1}{n} H(\pi_{\textbf{j}}\pi_n \nu, \mathcal{L}_n)=\alpha.$$
Since $H(\theta_n^{\textbf{j}}, \mathcal{L}_0^\#)=H(\pi_{\textbf{j}}\pi_n \nu, \mathcal{L}_n)$, the lemma follows.
\end{proof}

\begin{lemma}\label{lem:thetajLcn}
There exists $C\in \mathbb{Z}_+$ such that for each $\textbf{j}\in \Sigma$, we have 	$$\lim_{n\to\infty}\frac{1}{n}H\left(\theta^{\,\textbf{j}}_n,\mathcal{L}_{C n}^{\mathcal{X}}\right)=\frac{\log b}{\log(1/\lambda)}.$$
\end{lemma}
\begin{proof}
By Lemma~\ref{lem:constantC}, there exists $C\in \mathbb{Z}_+$ such that for all $n\ge 1$ and any two distinct $\textbf{i}, \textbf{k}\in \varLambda^{\hat{n}}$, $\pi_{\textbf{j}} g_{\textbf{i}}$ and $\pi_{\textbf{j}}g_{\textbf{k}}$ lie in distinct elements of $\mathcal{L}_{Cn}^\#$. Therefore $H(\theta^{\textbf{j}}_n, \mathcal{L}_{Cn}^\#)=\hat{n}\log b$. Since $\lim_{n\to\infty} n/\hat{n}=\log_b 1/\lambda$, the lemma follows.
\end{proof}

From now on, we fix $\textbf{j}\in \Sigma$ so that the conclusion of Lemma~\ref{lem:thetanL0} holds.
We shall write $\theta_n=\theta_n^{\textbf{j}}$.
Let
\begin{equation}\label{eqn:eps0}
\eps_0=\frac{1}{C}\left(\frac{\log b}{\log \frac{1}{\lambda}}-\alpha\right)>0.
\end{equation}

\subsection{Decomposition of entropy}
In the following lemma, we decompose the entropy of $\theta_n$ and $\pi_{\textbf{j}}\mu$ into small scales.  

\begin{lemma}\label{lem:thetadec}
For any $\tau>0$, there exists $C_0(\tau)>0$ such that if $k,n$ are positive integers with $n> C_0(\tau)k$, then
\begin{equation}\label{eqn:thetadec}
\frac{1}{Cn}H(\theta_n, \mathcal{L}_{Cn}^{\mathcal{X}}|\mathcal{L}_0^{\mathcal{X}})\le \mathbb{E}_{0\le i<Cn}^{\theta_n} \left[\frac{1}{k}H((\theta_n)_{\Psi,i}, \mathcal{L}_{i+k}^{\mathcal{X}})\right]+\tau,
\end{equation}
\begin{equation}\label{eqn:thetamudec}
\frac{1}{Cn} H(\pi_{\textbf{j}}\mu, \mathcal{L}_{(C+1)n}|\mathcal{L}_{n})\ge \mathbb{E}^{\theta_n}_{0\le i<Cn}\left[\frac{1}{k} H((\theta_n)_{\Psi,i}.\mu, \mathcal{L}_{i+k+n}|\mathcal{L}_{i+n})\right]-\tau.
\end{equation}
\end{lemma}
\begin{proof}
Using Lemma~\ref{lem:boundspart} and arguing in the same way of \cite[Lemma 3.4]{hochman2014self}, we have
$$\frac{1}{Cn}H(\theta_n, \mathcal{L}_{Cn}^{\mathcal{X}}|\mathcal{L}_0^{\mathcal{X}})= \mathbb{E}^{\theta_n}_{0\le i<Cn} \left[\frac{1}{k}H((\theta_n)_{\Psi,i}, \mathcal{L}_{i+k}^{\mathcal{X}})\right]+O\left(\frac{k}{n}\right).$$
Therefore, when $k$ and $n/k$ are large enough, (\ref{eqn:thetadec}) holds.
Similarly, we also have
$$\frac{1}{Cn} H(\pi_{\textbf{j}}\mu, \mathcal{L}_{(C+1)n}|\mathcal{L}_{n})\ge \frac{1}{Cn}\sum_{0\le i<Cn}\left[\frac{1}{k} H(\pi_{\textbf{j}}\mu, \mathcal{L}_{i+k+n}|\mathcal{L}_{i+n})\right]-\tau.$$
Note that $\pi_{\textbf{j}}\mu=(\theta_n).\mu$. By concavity of conditional entropy, we have
$$H(\pi_{\textbf{j}}\mu, \mathcal{L}_{i+k+n}|\mathcal{L}_{i+n})=H((\theta_n).\mu,\mathcal{L}_{i+k+n}|\mathcal{L}_{i+n})\ge \mathbb{E}^{\theta_n}(H((\theta_n)_{\Psi,i}.\mu, \mathcal{L}_{i+k+n}|\mathcal{L}_{i+n})).$$ Thus (\ref{eqn:thetamudec}) holds.
\end{proof}

\subsection{Proof of Theorem B'}
To conclude the proof of Theorem B', we shall further decompose the entropy
$$Q_{\Psi_0, i,n,k}:=\frac{1}{k} H([(\theta_n)_{\Psi_0, i}].\mu, \mathcal{L}_{i+k+n}|\mathcal{L}_{i+n})$$
into smaller scales and compare it with
$$\widetilde{Q}_{\Psi_0,i,n,k}=
\frac{1}{b^{\hat{i}}}\sum_{\textbf{u}\in \varLambda^{\hat{i}}}\int_{\mathcal{X}} \frac{1}{k} H(\Psi g_{\textbf{u}}\mu, \mathcal{L}_{i+k+n}) d(\theta_{n})_{\Psi_0,i}(\Psi),$$
for each $\Psi_0$ in the support of $\theta_n$.

\begin{lemma}\label{lem:entlow}
For any $\tau>0$, the following holds provided that $k\ge K_1(\tau)$: For any $\Psi_0$ in the support of $\theta_n$,
\begin{equation}\label{eqn:etacond}
Q_{\Psi_0, i, n, k}\ge
\widetilde{Q}_{\Psi_0, i, n,k}-\tau.
\end{equation}
\end{lemma}
\begin{proof}
By concavity of conditional entropy, the left hand side of (\ref{eqn:etacond}) is at least
$$\frac{1}{b^{\hat{i}}}\sum_{\textbf{u}\in \varLambda^{\hat{i}}}\int_{\mathcal{X}}\left(\frac{1}{k} H(\Psi g_{\textbf{u}}\mu, \mathcal{L}_{i+k+n}|\mathcal{L}_{i+n})\right)d\eta(\Psi),$$
where $\eta=(\theta_n)_{\Psi_0, i}$.
For each $\Psi$ in the support of $\eta$ and each $\textbf{u}\in \varLambda^{\hat{i}}$, the measure $\Psi g_{\textbf{u}}\mu$ is supported in an interval of length $O(b^{-(i+n)})$, hence
$H(\Psi g_{\textbf{u}} \mu, \mathcal{L}_{i+n})$ is uniformly bounded. The lemma follows.
\end{proof}

The following lemma will be proved in the next section, using Hochman's criterion on entropy increase.
\begin{lemma}[Entropy Increasing]\label{lem:entinc}
Assume $\alpha<1$. For every $\eps>0$,
there exist $\delta_*(\eps)>0$ and $K_2(\eps)>0$ such that for each $k\ge K_2(\eps)$ there exists $I_2(k,\eps)$ with the following property. Assume $i\ge I_2(k,\eps)$.
If $\Psi_0$ is in the support of $\theta_n$ and
$$\frac{1}{k}H((\theta_n)_{\Psi_0,i}, \mathcal{L}_{i+k}^{\mathcal{X}})\ge \eps,$$
then
$$Q_{\Psi_0,i,n,k}
\ge \widetilde{Q}_{\Psi_0,i,n,k}+\delta_*(\eps).$$
\end{lemma}


\begin{lemma}\label{lem:sumhatQ}
For any $\tau>0$, $k\ge K_3(\tau)$ and $n\ge N_3(\tau,k)$, the following holds:
$$\mathbb{E}^{\theta_n}_{0\le i<Cn} (\widetilde{Q}_{\Psi, i,n,k}) > (\alpha-\tau)(1-\tau).$$
\end{lemma}
\begin{proof}
First, we notice that
\begin{align*}
\mathbb{E}^{\theta_n^{\textbf{j}}}_{0\le i<Cn} (\widetilde{Q}_{\Psi, i, n,k})& =\frac{1}{Cn}\sum_{0\le i<Cn}
\frac{1}{b^{\hat{i}}}\sum_{\textbf{u}\in\varLambda^{\hat{i}}}\int_{\mathcal{X}} \frac{1}{k}H(\Psi g_{\textbf{u}}\mu, \mathcal{L}_{i+n+k})d\theta_n^{\textbf{j}}(\Psi)\\
&= \frac{1}{Cn} \sum_{0\le i<Cn} \frac{1}{b^{\hat{i}}}\sum_{\textbf{u}\in \varLambda^{\hat{i}}}\frac{1}{b^{\hat{n}}}\sum_{\textbf{v}\in \varLambda^{\hat{n}}}\frac{1}{k} H(\pi_{\textbf{j}}g_{\textbf{v}}g_{\textbf{u}}\mu, \mathcal{L}_{i+n+k})\\
&= \frac{1}{Cn}\sum_{0\le i<Cn}\frac{1}{b^{\hat{i}+\hat{n}}}\sum_{\textbf{w}\in \varLambda^{\hat{i}+\hat{n}}} \frac{1}{k}H(\pi_{\textbf{j}}g_{\textbf{w}}\mu,\mathcal{L}_{i+n+k}).
\end{align*}
By Lemma~\ref{lem:boundsinm}, for each $k\ge M(\tau/2)$, the following holds for all $n$ large enough:
$$\inf_{\textbf{j}\in \Sigma} \nu^{\hat{i}+\hat{n}} \left(\left\{\textbf{w}\in \varLambda^{\hat{i}+\hat{n}}: \frac{1}{k}H(\pi_{{\textbf{w}}^*\textbf{j}}\mu, \mathcal{L}_k)>\alpha-\tau/2 \right\}\right)>1-\tau.$$
By Lemma~\ref{TransformA} and Lemma~\ref{lem:affinetransform}, for $\textbf{w}\in \varLambda^{\hat{i}+\hat{n}}$, $|H(\pi_{{\textbf{w}}^*\textbf{j}}\mu,\mathcal{L}_k)-H(\pi_{\textbf{j}}g_{\textbf{w}}\mu, \mathcal{L}_{i+n+k})|$ is uniformly bounded. So when $k$ is large enough, the above displayed inequality implies that
 $$\inf_{\textbf{j}\in \Sigma} \nu^{\hat{i}+\hat{n}} \left(\left\{\textbf{w}\in \varLambda^{\hat{i}+\hat{n}}: \frac{1}{k}H(\pi_{\textbf{j}}g_{\textbf{w}}\mu, \mathcal{L}_{i+n+k})>\alpha-\tau \right\}\right)>1-\tau.$$
The lemma follows. 
\end{proof}
\begin{proof}[Proof of Theorem B'] Arguing by contradiction, assume that $\alpha<1$.
Let $\eps_0$ be given by (\ref{eqn:eps0}) and $\eps=\eps_0/2$. Let $\delta_*=\delta_*(\eps_0/2)$ be given by Lemma~\ref{lem:entinc} and let $\tau\in (0,\delta_*)$ be a small constant to be determined.
Fix $$k\ge \max(K_1(\tau), K_2(\eps), K_3(\tau)),$$ where $K_1(\tau)$ is given by Lemma~\ref{lem:entlow}, $K_2(\eps)$ is given by Lemma~\ref{lem:entinc} and $K_3(\tau)$ is given by Lemma~\ref{lem:sumhatQ}. Assume that $n$ is large enough. Then the left hand side of (\ref{eqn:thetadec}) tends to $\eps_0>0$.  By Lemma~\ref{lem:boundspart}, for any $i\ge 0$, any $\mathcal{L}_i^{\mathcal{X}}$-component $\eta$ of $\theta_n$, $\frac{1}{k} H(\eta, \mathcal{L}_{i+k}^{\mathcal{X}})$ is bounded from above by a constant.  Thus
$$\xi_0:=\mathbb{P}^{\theta_n}_{0\le i<Cn} \left(\frac{1}{k}H((\theta_n)_{\Psi,i}, \mathcal{L}_{i+k}^{\mathcal{X}})>\eps\right)$$
is bounded from below by a positive constant $2p$. By Lemma~\ref{lem:entinc},
$$\xi:=\mathbb{P}^{\theta_n}_{0\le i<Cn} \left(
Q_{\Psi,i,n,k}> \widetilde{Q}_{\Psi, i, n,k}
+\delta_*\right)\ge \xi_0-\frac{I_2(k,\eps)}{Cn}\ge p.$$
Therefore, by Lemmas~\ref{lem:sumhatQ} and \ref{lem:entlow},
$$
\mathbb{E}^{\theta_n}_{0\le i<Cn} \left(Q_{\Psi, i, n,k}
\right)
\ge \mathbb{E}^{\theta_n}_{0\le i<Cn} \left(\widetilde{Q}_{\Psi, i, n,k}\right)+\xi \delta_*-(1-\xi)\tau
\ge (\alpha-\tau)(1-\tau)+\xi\delta_*-(1-\xi)\tau.
$$

Choosing $\tau>0$ small enough, we obtain
$$\mathbb{E}^{\theta_n}_{0\le i<Cn} \left(Q_{\Psi, i,n,k}\right)\ge \alpha+ p\delta_*/2.$$
However, as $n\to\infty$, the left hand side of (\ref{eqn:thetamudec}) converges to $\alpha$, a contradiction!
\end{proof}

\subsection{Proof of the Entropy Increasing Lemma}
In the rest of this section, we shall prove Lemma~\ref{lem:entinc}. The following is a version of Hochman's entropy increasing criterion, see  ~\cite[Theorem 2.8]{hochman2014self} and~\cite[Theorem 4.1]{barany2019hausdorff}.
\begin{theorem}[Hochman]\label{thm:hochmanentgrow}
	For any $\varepsilon>0$ and $m\in\mathbb{Z}_+$ there exists $\delta=\delta(\varepsilon,m)>0$ such that for
	$k>K(\varepsilon,\delta,m)$, $n\in\mathbb{N}$, and $\tau,\theta\in\boldsymbol{\mathscr{P}}(\mathbb{R})$,
	if
\begin{enumerate}
\item[(1)] $\text{diam}(\text{supp}(\tau)),\text{diam}(\text{supp}(\theta))\le b^{-n}$,
\item[(2)] $\tau$ is $(1-\varepsilon,\frac{\varepsilon}2,m)$-entropy porous from scales $n$ to $n+k$,
\item [(3)]$\frac{1}{k}H(\theta,\mathcal{L}_{n+k})>\varepsilon$,
\end{enumerate}	
	then
	$$\frac{1}{k}H(\theta\ast\tau,\mathcal{L}_{n+k})\ge\frac{1}{k}H(\tau,\mathcal{L}_{n+k})+\delta,$$
where $\ast$ denotes the convolution. 
\end{theorem}

For $\eta:=(\theta_n)_{\Psi_0, i}$ as in Lemma~\ref{lem:entinc}, we decompose $\eta.\mu$ as follows:
$$\eta.\mu=\frac{1}{b^{\hat{i}}} \sum_{\textbf{u}\in \varLambda^{\hat{i}}} \eta. g_{\textbf{u}}\mu.$$
We first show that the entropy of each term in the right hand side can be represented by entropy of convolutions of line measures.
\begin{lemma}\label{lem:dot2convolution}
There is a constant $C_1>0$ and for each $\tau>0$ there exists $K(\tau)$ such that when $i\ge C_1k$, $k\ge K(\tau)$ the following holds:
$$\left|\frac{1}{k} H(\eta. g_{\textbf{u}}\mu,\mathcal{L}_{i+k+n}|\mathcal{L}_{i+n}) - \frac{1}{k}H((\eta.\delta_{g_{\textbf{u}}(0)})\ast (\Psi_0 g_{\textbf{u}}\mu),\mathcal{L}_{i+k+n})\right| <\tau.$$
\end{lemma}
\begin{proof}
Write $z_0:=g_{\textbf{u}}(0)=(x_0,y_0)$.
Define $F, G: \text{supp}(\eta) \times \text{supp}(g_{\textbf{u}}\mu)\to \mathbb{R}$ by
$$F(\Psi, z)=\Psi(z),\,\,\, G(\Psi,z)=\Psi(z_0)+\Psi_0(z)-\Psi_0(z_0).$$
Note that
$F(\eta\times g_{\textbf{u}}\mu)=\eta.g_{\textbf{u}}\mu$
and $G(\eta\times g_{\textbf{u}}\mu)$ is a translation of the convolution of $\eta.\delta_{z_0}$ and $\Psi_0. g_{\textbf{u}}\mu$. By Lemma~\ref{lem:constantR}, $\eta.\delta_{z_0}$ is supported in an interval of length $O(b^{-(i+n)})$. The same is also true for $\Psi_0. g_{\textbf{u}}\mu$, and hence for the measure $G(\eta\times g_{\textbf{u}}\mu)$. It follows that $H(G(\eta\times g_{\textbf{u}}\mu), \mathcal{L}_{i+n})$ is bounded from above by a constant.
Thus it is enough  to show that
$$F(\Psi,z)-G(\Psi,z)=O(b^{-(i+k+n)})$$
under the assumption that $i/k$ is large enough.


To this end, write $\Psi(x,y)=\lambda^{\hat{n}}(y-\Gamma_{\textbf{v}}(x))+c$ and $\Psi_0(x,y)=\lambda^{\hat{n}}(y-\Gamma_{\textbf{v}_0}(x))+c_0$.
Then for $z=(x,y)$, we have
$$\big|F(\Psi,z)-G(\Psi,z)\big|=\lambda^{\hat{n}}\big|\int_x^{x_0}(Y_{\textbf{v}}-Y_{\textbf{v}_0})(s)ds\big|=b^{-n}\cdot O(|x-x_0|).$$
Note that $|x-x_0|\le b^{-\hat{i}}=O(b^{-\frac{\log b}{\log {\sfrac1{\lambda}}} i})$. So when $i/k$ is sufficiently large, $|x-x_0|=O(b^{-(i+k)})$, and hence $\left|F(\psi,z)-G(\psi,z)\right|=O(b^{-(i+k+n)})$.
\end{proof}

The measure $\eta.\delta_{g_{\textbf{u}}(\textbf{0})}$ plays the role of $\theta$, and $\Psi_0 g_{\textbf{u}}\mu$ plays the role of $\tau$ in Hochman's theorem. Lemma~\ref{lem:entinceta} shows that for a definite amount of $\textbf{u}$, $\eta.\delta_{g_{\textbf{u}}(\textbf{0})}$ has definite entropy.

\begin{proof}[Proof of Lemma~\ref{lem:entinc}]
First, by concavity of conditional entropy,
$$\frac{1}{k} H(\eta.\mu, \mathcal{L}_{i+k+n}|\mathcal{L}_{i+n})\ge b^{-\hat{i}}\sum_{\textbf{u}\in \Lambda^{\hat{i}}} \frac{1}{k}H(\eta. g_{\textbf{u}}\mu, \mathcal{L}_{i+k+n}|\mathcal{L}_{i+n}).$$
By Lemma~\ref{lem:dot2convolution}, for any $\tau>0$,
\begin{equation}\label{eqn:etamu2conv}
\frac{1}{k}H(\eta.\mu, \mathcal{L}_{i+k+n}|\mathcal{L}_{i+n})\ge \frac{1}{b^{\hat{i}}}\sum_{\textbf{u}\in \Lambda^{\hat{i}}}\frac{1}{k} H((\eta.\delta_{g_{\textbf{u}}(\textbf{0})})\ast (\Psi g_{\textbf{u}}\mu),\mathcal{L}_{i+k+n})-\tau
\end{equation}
holds for each $\Psi$ in the support of $\eta$, provided that $k$ is large enough and $i\ge C_1k$.
By \cite[Corollary 4.10]{hochman2014self}, increasing $K(\tau)$ if necessary, we have
\begin{equation}\label{eqn:lowbd}
\frac{1}{k} H((\eta.\delta_{g_{\textbf{u}}(\textbf{0})})\ast (\Psi g_{\textbf{u}}\mu),\mathcal{L}_{i+k+n})\ge \frac{1}{k} H(\Psi g_{\textbf{u}}\mu, \mathcal{L}_{i+k+n})-\tau,
\end{equation}
for any $\Psi$ and $\textbf{u}$.

Next, let us prove the following

{\bf Claim.} There exist $p, \delta_o>0$ and for each $k$ large enough, there exists $I(\eps, k)$ such that the following holds when $i\ge I(\eps,k)$. For each $\Psi\in \text{supp}(\eta)$, there is a subset $\Omega^{\Psi}$ of $\varLambda^{\hat{i}}$ with $\nu^{\hat{i}}(\Omega^\Psi)>p$ such that for $\textbf{u}\in \Omega^\Psi$, we have an entropy growth:
\begin{equation}\label{eqn:entinceta}
\frac{1}{k}H((\eta.\delta_{g_{\textbf{u}}(\textbf{0})})\ast (\Psi g_{\textbf{u}}\mu),\mathcal{L}_{i+k+n})\ge\frac1k H(\Psi g_{\textbf{u}}\mu,\mathcal{L}_{i+k+n})+\delta_o.
\end{equation}

Take $\xi=\min(1-\alpha, \delta_*,p)$, where $\delta_*=\delta_*(\eps)$ and $p=p(\eps)$ are as in Lemma~\ref{lem:entinceta}. So the set
$$\Omega_0=\left\{\textbf{u}\in \varLambda^{\hat{i}}: \frac{1}{k}H(\eta. \delta_{g_{\textbf{u}}(\textbf{0})},\mathcal{L}_{i+k+n})>\xi\right\}$$
satisfies $\nu^{\hat{i}}(\Omega_0)>p$, provided that $i, k$ are large enough.
%
By Theorem~\ref{thm:entporous}, there exists $m$, and for each $k$ large enough there exists $I_k$ such that when $i\ge I_k$, for any $\Psi$ in the support of $\eta$, we have
$\nu^{\hat{i}}(\Omega^{\Psi}_1)>1-\frac{\xi}{2},$
where
$$\Omega^{\Psi}_1=\{\textbf{u}\in \varLambda^{\hat{i}}: \Psi g_{\textbf{u}}\mu \mbox{ is } (\alpha, \xi/2, m)-\mbox{ entropy porous from scale } n+i \mbox{ to } n+k+i\}. $$
Thus $\nu^{\hat{i}}(\Omega^{\Psi})\ge p/2$, where $\Omega^{\Psi}=\Omega_1^{\Psi}\cap \Omega_0$.
As we have seen before, for any $\textbf{u}\in \varLambda^{\hat{i}}$, the measures $\eta.\delta_{g_{\textbf{u}}(\textbf{0})}$ and $\Psi g_{\textbf{u}}\mu$ are supported in intervals of length $O(b^{-(i+n)})$. 
Applying Theorem~\ref{thm:hochmanentgrow} (with $\xi$ in the place of $\eps$, $i+n$ in the space of $n$), we complete the proof of the claim.

Let us now complete the proof of Lemma~\ref{lem:entinc}. By (\ref{eqn:lowbd}) and (\ref{eqn:entinceta}), we have
$$\frac{1}{b^{\hat{i}}} \sum_{\textbf{u}\in \varLambda^{\hat{i}}} \frac{1}{k} H(\eta. g_{\textbf{u}}\mu, \mathcal{L}_{i+k+n})\ge \frac{1}{b^{\hat{i}}}\sum_{\textbf{u}\in \varLambda^{\hat{i}}} \frac{1}{k}H(\Psi g_{\textbf{u}}\mu, \mathcal{L}_{i+k+n})+ p\delta_o-\tau.$$
By (\ref{eqn:etamu2conv}), this gives us
$$\frac{1}{k}H(\eta.\mu, \mathcal{L}_{i+k+n}|\mathcal{L}_{i+n})\ge \frac{1}{b^{\hat{i}}}\sum_{\textbf{u}\in \varLambda^{\hat{i}}} \frac{1}{k}H(\Psi g_{\textbf{u}}\mu, \mathcal{L}_{i+k+n})+ p\delta_o-2\tau.$$
Integrating over $\Psi$ with respect to $\eta$ gives us
$$\frac{1}{k}H(\eta.\mu, \mathcal{L}_{i+k+n}|\mathcal{L}_{i+n})\ge \widehat{Q}_{\Psi_0, i, n,k}
+\frac{p\delta_o}{2},$$
provided that we had chosen $\tau$ small enough.
\end{proof}
\bibliographystyle{plain}             

\end{document}